\documentclass[12pt]{amsart}

\pdfoutput=1

\usepackage{latexsym,amsmath,amssymb,amsmath,mathtools,tikz,microtype,comment}
\usepackage[english]{babel}
\usepackage[margin=2.5cm]{geometry}
\usepackage{enumitem}

\usepackage[colorlinks=true, citecolor=blue, urlcolor=magenta, breaklinks=true]{hyperref}

\numberwithin{equation}{section}
\newtheorem{theorem}{Theorem}[section]
\newtheorem{lemma}[theorem]{Lemma}

\newtheorem{corollary}[theorem]{Corollary}

\newtheorem{question}[theorem]{Question}

\theoremstyle{definition}
\newtheorem{definition}[theorem]{Definition} 
\newtheorem{remark}[theorem]{Remark}
\newtheorem{example}[theorem]{Example}
\newtheorem{construction}[theorem]{Construction}

\DeclareMathOperator{\rank}{rank}
\DeclareMathOperator{\supp}{supp}
\DeclareMathOperator{\pdim}{pdim}
\DeclareMathOperator{\reg}{reg}

\newcommand{\K}{{\mathbb{K}}}
\newcommand{\N}{{\mathbb{N}}}
\newcommand{\R}{{\mathbb{R}}}
\newcommand{\Z}{{\mathbb{Z}}}
\newcommand{\Q}{{\mathbb{Q}}}
\begin{document}

%%%%%%%%%%%%%%%%%%%%%%%%%%%%%%%%%%%%%%%%%%%%%%%%%%%%%%%%%%%%%%%% 
 
\title{Splittings of Toric Ideals}
\thanks{\today}

\author[G. Favacchio]{Giuseppe Favacchio}
\author[J. Hofscheier]{Johannes Hofscheier}
\author[G. Keiper]{Graham Keiper}
\author[A. Van Tuyl]{Adam Van Tuyl}

\address[G. Favacchio]
	{ DISMA-Department of Mathematical Sciences, Politecnico di Torino, Turin, Italy}
	\email{giuseppe.favacchio@polito.it}

\address[J. Hofscheier]{School of Mathematical Sciences\\
University of Nottingham\\
Nottingham, NG7 2RD\\
UK}
\email{johannes.hofscheier@nottingham.ac.uk}

\address[G. Keiper, A. Van Tuyl]{Department of Mathematics and Statistics\\
McMaster University, Hamilton, ON, L8S 4L8}
\email{keipergt@mcmaster.ca, vantuyl@math.mcmaster.ca}

 
\begin{abstract}
  Let $I \subseteq R = \K[x_1,\ldots,x_n]$ be a toric ideal,
  i.e., a binomial prime ideal. We investigate when
  the ideal $I$ can be ``split" into the sum of two smaller
  toric ideals.  For a general toric ideal $I$,
  we give a sufficient condition for this splitting
  in terms of the integer matrix that defines $I$.
  When $I = I_G$ is the toric ideal of a finite simple graph $G$,
  we give additional splittings of $I_G$ related to subgraphs
  of $G$.  When there exists a splitting $I = I_1+I_2$ of the toric ideal,
  we show that in some cases we can describe
  the (multi-)graded Betti numbers of $I$ in terms
  of the (multi-)graded Betti numbers of $I_1$ and $I_2$.
\end{abstract}
 \maketitle

%%%%%%%%%%%%%%%%%%%%%%%%%%%%%%%%%%%%%%%%%%%%%%%%%%%%%%%%%%%%%%%%%

\section{Introduction}

Toric ideals appear in the intersection of many areas of mathematics,
including commutative algebra, algebraic geometry, combinatorics, and have
applications to many areas, e.g., algebraic statistics \cite{DS}.
A \emph{toric ideal} $I$ in a  polynomial ring $R = \K[x_1,\ldots,x_n]$
(with $\K$ an algebraically closed field of characteristic zero)
is a prime ideal generated by 
binomials.   For detailed introductions to toric ideals, we refer
the readers to \cite{CoxLittleSchecnk,EH,Binomi,St}.

Under some mild assumptions, a toric ideal $I \subseteq R$ is a
(multi-)homogeneous ideal, and consequently,
one can compute its {\em (multi-)graded Betti numbers}, that is,
\[\beta_{i,j}(I) = \dim_\K {\rm Tor}_i^R(I, \K)_j,\]
where $j \in \N$ or $j\in \N^n$, depending upon
our grading.  Betti numbers are examples of the
homological invariants
of $I$ that are encoded into the minimal (multi-)graded 
free resolution of $I$.
It was shown by
Campillo and Marijuan \cite{CM} and Campillo and Pison \cite{CP},
and independently by
Aramova and Herzog
\cite{AH},
that one can compute the 
multi-graded Betti numbers of any multi-homogeneous toric ideal by computing
the ranks of reduced simplicial homology groups
(see \cite[Theorem 67.5]{PeevaBook} and \cite[Theorem 12.12]{St}).
This result is a toric ideal analog of the well-known
Hochster's Formula (e.g., \cite[Theorem 3.31]{Binomi}) for monomial ideals.
Applying these formulas, however, to compute the
(multi-)graded Betti numbers can be a formidable task.

One current stream of research has been interested in these homological
invariants under the additional assumption that $I = I_G$ is the
\emph{toric ideal of finite simple graph $G$}.
Specifically, given a finite simple graph $G$ with vertex
set $V(G) = \{x_1,\ldots,x_n\}$ with edge set $E(G) = \{e_1,\ldots,e_q\}$,
the toric ideal $I_G$ is the kernel of the map
$\varphi\colon \K[e_1,\ldots,e_q] \rightarrow \K[x_1,\ldots,x_n]$
given by $\varphi(e_i) = x_{i,1}x_{i,2}$ where $e_i = \{x_{i,1},x_{i,2}\}$
(see Section \ref{s. background}).  One is then interested in
relating the homological invariants of $I_G$ to the graph theoretical
invariants of $G$.  As examples of this approach,
\cite{OH,TT,VRees} relate the
generators of $I_G$ to walks in $G$,
\cite{BOVT,CN,HBOK} give graph theoretical bounds
on the regularity and projective dimension for the toric
ideals of some families of graphs, \cite{GM} investigates the 
$N_p$-property of the toric ideals of bipartite graphs, \cite{HMT,T} 
studies when the toric ideal of a bipartite graph has a linear resolution,
\cite{GHKKVTP,NN} compute all the graded Betti
numbers of $I_G$ for specific families of graphs, and
\cite{GV,HHKOK} relate the invariants of depth and
multiplicity of $R/I_G$ to $G$.  

Given this interest in homological invariants, it is natural
to ask when one can compute the Betti numbers
of toric ideals using recursive or inductive methods.
With this goal in mind, we investigate when one can
``split" the
toric ideal into ``smaller" toric ideals. More 
precisely, we say a toric ideal $I$ has {\em toric splitting} (or $I$ is a {\em splittable toric ideal}) if there exists toric ideals $I_1$ and $I_2$ such that $I = I_1 + I_2$. Our main motivation is to identify 
toric splittings of $I$ so that the graded Betti 
numbers of $I$ can be computed in terms of those of $I_1$ and $I_2$,
thus complementing existing approaches to computing these invariants.
We were also
inspired by \cite{FHVT}
which considered splittings of
monomial ideals to compute (or bound)
the Betti numbers.

One immediately encounters the following obstacle:
Suppose the toric ideal $I = \langle f_1,\ldots,f_t \rangle$ is
minimally generated by the binomials $V = \{f_1,\ldots,f_t\}$. Given a
non-trivial partition of the generators, say 
$V = W \sqcup Y$, the ideals $I_1 = \langle g ~|~ g \in W \rangle$ and
$I_2 = \langle g ~|~ g \in Y \rangle$ are binomial
ideals, but these ideals may fail to be prime. Hence, toric splittings
may not even exist!

The main results of this paper were inspired by the following
prototypical example of a toric splitting.
Given a graph $G$ and cycle $C$ of even length $2d$,
consider the graph $H$ which is formed by identifying any edge of $G$
with an edge of $C$ (see Figure \ref{fig:onecycle}).
 \begin{figure}[!ht]
    \centering
    \begin{tikzpicture}[scale=0.325]
      % graph G
      \draw[dotted] (6,0) circle (3cm) node{$G$};
      \draw[thin,dashed] (4,1) -- (5,2);
      \draw[thin,dashed] (4,-1) -- (5,-2);
      \draw[dashed] (4,-1) -- (3.5,-2);
      \draw (4,1) -- (4,-1) node[midway, right] {$e$};
      \draw[dashed] (4,1) -- (3.5,2);

      % rest of cycle C
      \node at (0,0) {$C$};
      \draw[loosely dotted] (3.5,2) edge[bend right=10] (2,3.5);
      \draw[dashed] (1,4) -- (2,3.5);
      \draw (-1,4) -- (1,4);
      \draw[dashed] (-1,4) -- (-2,3.5);
      \draw[loosely dotted] (-2,3.5) edge[bend right=10] (-3.5,2);
      \draw[dashed] (-4,1) -- (-3.5,2);
      \draw (-4,1) -- (-4,-1);
      \draw[dashed] (-4,-1) -- (-3.5,-2);
      \draw[loosely dotted] (-2,-3.5) edge[bend right=-10] (-3.5,-2);
      \draw[dashed] (-1,-4) -- (-2,-3.5);
      \draw (-1,-4) -- (1,-4);
      \draw[dashed] (1,-4) -- (2,-3.5);
      \draw[loosely dotted] (3.5,-2) edge[bend right=-10] (2,-3.5);
    
      % draw nodes
      \fill[fill=white,draw=black] (-4,1) circle (.1);
      \fill[fill=white,draw=black] (-4,-1) circle (.1);
      \fill[fill=white,draw=black] (4,1) circle (.1);
      \fill[fill=white,draw=black] (4,-1) circle (.1);
      \fill[fill=white,draw=black] (-1,4) circle (.1);
      \fill[fill=white,draw=black] (1,4) circle (.1);
      \fill[fill=white,draw=black] (-1,-4) circle(.1);
      \fill[fill=white,draw=black] (1,-4) circle (.1);
    \end{tikzpicture}
    \caption{Connecting an even cycle  $C$ to a graph $G$ to make a graph $H$.}
    \label{fig:onecycle}
 \end{figure}
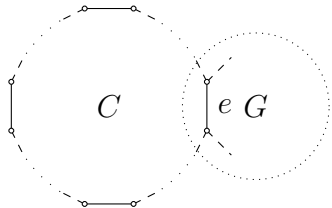
 In this case, the toric ideal of $H$ is splittable.  Specifically,
 $I_{H} = I_{G} + I_C$, and furthermore, the graded Betti numbers satisfy
  (Corollary \ref{c. firstglue})
 \[\beta_{i,j}(R/I_H) = \beta_{i,j}(R/I_G) + \beta_{i,j-d}(R/I_G) ~\mbox{for
   all $i,j \geq 0.$}\]
 We want to determine a  more general context where this
 example becomes a special case.

 In Section \ref{sec:split-tor-id} we considered toric ideals in general. 
 A toric ideal can be constructed from
 an $n \times s$ integer matrix $A$ (see Section \ref{s. background}).
 Our first main result (see Lemma \ref{lem:lattice-ideal}) gives a
 sufficient condition for a
 toric ideal $I$ to be a splittable toric ideal in terms of
 the matrix $A$.
 In fact, under the hypotheses of Lemma \ref{lem:lattice-ideal},
 one of the two ideals in the splitting will be a principal ideal.
 If the toric ideal is also multi-graded, we apply Lemma
 \ref{lem:lattice-ideal} to relate
 the graded Betti numbers of splittable ideal $I = I_1 +I_2$
 to those of $I_1$ and $I_2$.  When we specialize our Lemma
 \ref{lem:lattice-ideal} to the toric ideals of graphs, we recover
 the example described above.  Of independent interest, our Lemma \ref{lem:tech}
gives an ideal membership criterion for a particular binomial to belong
to a two-generated binomial ideal.

 In Section \ref{s. Splittings of Toric Ideals of Graphs} we restrict
 to splittings of toric ideals of graphs.  The results in this
 section are based upon the observation that the graph $H$ in
 Figure \ref{fig:onecycle} is formed by ``gluing'' an even cycle
 to an edge of a graph.  After formally defining ``gluing'' (and its
 inverse operation of ``splitting''), we generalize the above
 example by showing that if any bipartite graph $K$ is glued along
 an edge of a graph $G$ to form $H$, then
 $I_H = I_G + I_K$ is a splitting of toric ideals
 (see Corollary \ref{t. main splitting}).  Furthermore, Theorem
 \ref{t.tensor product} relates the graded Betti numbers of $I_G,I_K$,
 and $I_H$.   Our main result (Theorem \ref{thm:glue-Gs-along-path}) is more general in that we consider
 a gluing of $G$ and $K$ along a path; in this case $I_H$ is
 the sum of $I_G$ and $I_K$ up to a saturation by a monomial.

 Our paper is structured as follows. In
 Section  \ref{s. background} we recall the relevant definitions and results
 about toric ideals and graph theory. In Section \ref{sec:split-tor-id}  we 
 present our main technical lemma and consequences
 for the graded Betti numbers of toric ideals. 
 In Section \ref{s. Splittings of Toric Ideals of Graphs} we consider splittings of toric ideals
 of graphs and  the consequence of this splitting for the graded Betti numbers of such graphs. The last section suggests some future research directions.

{\bf Acknowledgements.} 
The authors thank Jason Brown and David Cox for
answering our questions. Our results were inspired by computer
calculations using {\em Macaualy2} \cite{M2}. Favacchio thanks
McMaster University for its hospitality and
the support of the Universit\`a degli Studi di
Catania ``Piano della Ricerca 2016/2018 Linea di intervento 2" and the
``National 
Group for Algebraic and Geometric Structures, and their Applications" (GNSAGA-INdAM). 
Hofscheier is supported by a Nottingham Research Fellowship from the University of Nottingham.
Van Tuyl acknowledges the support
of NSERC RGPIN-2019-05412.

%%%%%%%%%%%%%%%%%%%%%%%%%%%%%%%%%%%%%%%%%%%%%%%%%%%%%%%%%

\section{Notation and Background}
\label{s. background}

We recall the relevant definitions and background on toric ideals and toric ideals of graphs.

\subsection{Toric Ideals}
\label{subsec:tor-id}
Fix an integer $n \geq 1$, and let $\mathbf{e}_1,\ldots, \mathbf{e}_n$ denote the standard basis vectors
of $\R^n$ (or $\Z^n)$. The \emph{support} of a vector
$\alpha = (a_1,\ldots,a_n)$ in $\R^n$ (or $\Z^n$)
is\[ \supp(\alpha) \coloneqq \{ i=1, \ldots, n ~|~a_i \neq 0 \}\]  Any
$\alpha \in \R^n$ (or $\Z^n$) can be decomposed
uniquely as $\alpha = \alpha_+ - \alpha_-$ where
\[
  \alpha_+ = \sum_{a_i > 0} a_i{\bf e}_i \quad \mbox{and} \quad \alpha_- = \sum_{a_i < 0} (-a_i){\bf e}_i \text{.}
\]

Let $\{\alpha_1,\ldots,\alpha_s\} \subseteq \Z^n$, and set $A$ to be the $n \times s$ matrix $A = 
[ \alpha_1\; \cdots\; \alpha_s ]$. The matrix $A$ induces a map $\Z^s \to \Z^n$; in fact, we have an exact sequence
\[
  0 \to L \to \Z^s \to \Z^n \text{,}
\]
where $L$ is the kernel of $A$. Recall that $L$ is a \emph{lattice}, i.e., a finitely generated free abelian group. In particular, $L$ is isomorphic (as a $\Z$-module) to $\Z^t$ for some $t$. The notion of saturation is needed for the
proof of Theorem \ref{t. Splitting 1}. Let $M$ be a lattice and $L\subset M$ a sublattice. The lattice $L$ is \emph{saturated} in $M$ if for any $\ell \in M$ such that some positive integer multiple of $\ell$ is contained in $L$, then $\ell$ is already in $L$. Note that $L$ is saturated in $M$ if and only if $M/L$ is torsionfree.

\begin{definition}
Let $A = [ \alpha_1\; \cdots\; \alpha_s ]$ be an $n \times s$ matrix as above, and let $R = \K[x_1, \ldots, x_s]$. The {\em toric ideal} of $A$ is the ideal
\[
  I_{A} = \langle x^{\alpha_{+}}-x^{\alpha_{-}}~|~\alpha\in\text{ker}(A) \rangle \subseteq R \text{.}
\]
\end{definition}

\begin{remark}
  A toric ideal is sometimes defined as a binomial ideal (an ideal generated by binomials, that is, the difference of two terms) that is a prime ideal. It is clear from our definition that $I_A$ is a binomial ideal. To see that $I_A$ is a prime
  ideal, consider the polynomial ring $R = \K[x_1, \ldots, x_s]$ and the Laurent polynomial ring $S= \K[t_1, t_1^{-1}, \ldots, t_n, t_n^{-1}]$. Define a homomorphism of semigroups algebras $\varphi \colon R \to S$ by mapping 
  \[
    x_i \mapsto t^{\alpha_i} = t_1^{a_{i,1}}t_2^{a_{i,2}}\cdots t_n^{a_{i,n}} ~~\mbox{where $\alpha_i = (a_{i,1},\ldots,a_{i,n})$.}
  \]
  Then an equivalent definition (see \cite[Chapter 4]{St}) for the toric ideal of $A$ is $I_A = \ker \varphi$. Because the image of $\varphi$ is a domain, it follows that $I_A$ is prime. 
\end{remark}

Information about $R/I_A$ is encoded into the matrix. For example:
\begin{theorem}\cite[Proposition 3.1]{Binomi}
With $A$ as above, $\dim(R/I_A) = \rank(A)$.
\end{theorem}

Toric ideals are not necessarily homogeneous with respect to the standard grading of $R$, i.e., $\deg(x_i) = 1$ for $i=1,\ldots,s$, or even non-standard graded. Because our primary interest is the minimal graded free resolution of toric ideals, it is necessary to know when $I_A$ is a (multi-)homogeneous ideal. The next lemma captures when $I_A$ is standard graded.

\begin{lemma}\cite[Lemma 4.14]{St}
\label{homogtoric}
  Let $A = [ \alpha_1 \; \cdots \;\alpha_s ]$ be an $n \times s$ matrix with $\alpha_i \in \Z^n$. Then $I_A$ is a homogeneous ideal if and only if there exists a vector $c \in \Q^n$ such that $\alpha_i \cdot c =1$ for all $i=1,\ldots,n$. Here, $\alpha_i \cdot c$ denotes the standard Euclidean inner product.
\end{lemma}

If  $L \cap \N^s = \{{\bf 0}\}$, we can induce an $\N^n$-grading on $R$, $I_A$, and $R/I_A$, by setting $\deg x_i = \alpha_i$ for $i=1,\ldots,s$. For example,
if each column of $A$ belongs to $\N^n$, then the condition $L \cap \N^s = \{{\bf 0}\}$ is satisfied.
If $I_A$ is $\N^n$-graded, then there is a minimal
multi-graded free resolution of $I_A$, i.e.,
\[
  0 \to \bigoplus_{\alpha\in \N^n} R(-\alpha)^{\beta_{l,\alpha}(I_A)} \to \bigoplus_{\alpha\in \N^n} R(-\alpha)^{\beta_{l-1,\alpha}(I_A)} \to \cdots \to \bigoplus_{\alpha\in \N^n} R(-\alpha)^{\beta_{0,\alpha}(I_A)} \to I_A \to 0 \text{,}
\]
where $R(-\alpha)$ denotes the $\N^n$-grading of $R$ twisted by $-\alpha$, i.e., $R(-\alpha)_\gamma = R_{\gamma-\alpha}$ for all $\gamma \in \N^n$. The 
\emph{multi-graded Betti number}  $\beta_{i,\alpha}(I_A)$ is the number of minimal generators of the $i$-th syzygy module of $I_A$ of multidegree $\alpha \in \N^n$.  
Each $\beta_{i,\alpha}(I_A)$ is equal to the rank of a reduced simplicial homology group of a simplicial complex related to $\alpha$ (see \cite[Theorem 12.12]{St}).

If there exists an integer $d > 0$ such that every column $\alpha_i$
of $A$ satisfies $|\alpha_i| = \sum_{j=1}^n a_{ij} = d$, then the standard grading and the $\N^n$-grading of $I_A$ are compatible in the following sense:
\begin{equation}\label{multi-standard-grading}
  \beta_{i,j}(I_A) = \sum_{|\alpha|=d\cdot j} \beta_{i,\alpha}(I_A).
\end{equation}

The theme of this paper is to understand when $I_A$ can be ``split'' into smaller toric ideals. The following result, which is undoubtedly known, describes one case in which $I_A$ is splittable.

\begin{lemma}\label{lem:disjoint-splitting}
  Let $A_1, \ldots, A_k$ be matrices with integer entries of dimensions
  $n_i \times s_i$ ($i=1,\ldots, k$) and
  set $R = \K[x_{1,1},\ldots,x_{1,s_1},\ldots,x_{k,1},\ldots,x_{k,s_k}]$.
  Consider the block matrix
  \[
    A = \begin{bsmallmatrix}
      A_1 & & 0 \\
      &\raisebox{0pt}{$\scalebox{.75}{$\ddots$}$} & \\
      0 &&A_k \\
    \end{bsmallmatrix}
    \in \Z^{(n_1+\cdots+n_k) \times (s_1 + \cdots + s_k )}\text{.}
  \]
  Then
  \[
    I_{A} = I_{A_1} + \cdots + I_{A_k} \subseteq R
  \]
  where $I_{A_i}$ is the toric ideal of ${A_i}$, but viewed as an ideal in $R$.
\end{lemma}

\begin{proof}
  For each $i=1,\ldots,k$, set $R_i = \K[x_{i,1},\ldots,x_{i,s_i}]$. Let $\beta \in {\rm ker}(A_i)$. So $x^{\beta_+} - x^{\beta_-} \in I_{A_i} \subseteq R_i$. But then
  \[
    \gamma = (\underbrace{0,\ldots,0}_{s_1+\cdots+s_{i-1}},\beta,\underbrace{0,\ldots,0}_{s_{i+1}+\cdots +s_k}) \in \ker(A) \text{.}
  \]
  So $x^{\beta_+} - x^{\beta_-} = x^{\gamma_+}- x^{\gamma_-} \in I_A$. Thus $I_{A_i} \subseteq I_A$, if $I_{A_i}$ is viewed as an ideal of $R$. Consequently, $I_{A_1}+\cdots + I_{A_k} \subseteq I_A$.

  For the reverse inclusion, we do induction on $k$, where our base case is $k =2$. Suppose that $\alpha \in {\rm ker}(A) \subseteq \Z^{s_1+s_2}$. Write $\alpha$ as $\alpha = (\beta,\gamma)$ where $\beta \in \Z^{s_1}$ and $\gamma \in \Z^{s_2}$. Then
  \begin{eqnarray*}
    x^{\alpha_+}-x^{\alpha_-} & = &x^{\beta_+}x^{\gamma_+}- x^{\beta_-}x^{\gamma _-} 
     = x^{\gamma_+}(x^{\beta_+} - x^{\beta_-}) +
      x^{\beta_-}(x^{\gamma_+} - x^{\gamma_-}).
  \end{eqnarray*}
  But $x^{\beta_+} - x^{\beta_-} \in I_{A_1}$ and $x^{\gamma_+} - x^{\gamma_-} \in I_{A_2}$. So $I_{A} \subseteq I_{A_1} + I_{A_2}$.

  Now suppose that $k > 2$, and let $\alpha \in {\rm ker}(A) \subseteq \Z^{s_1+\cdots+s_k}$. Write $\alpha$ as $(\beta,\gamma)$ with $\beta \in \Z^{s_1+\cdots+s_{k-1}}$ and $\gamma \in \Z^{s_k}$. As above,
  \[
    x^{\alpha_+}-x^{\alpha_-} = x^{\gamma_+}(x^{\beta_+} - x^{\beta_-}) + x^{\beta_-}(x^{\gamma_+} - x^{\gamma_-}) \text{.}
  \]
  By induction, $x^{\beta_+}-x^{\beta_-} \in I_{A_1}+\cdots + I_{A_{k-1}}$, while $x^{\gamma_+} - x^{\gamma_-} \in I_{A_k}$. The result now holds.
\end{proof}

\begin{theorem}\label{kunneth}
  With the notation and hypotheses of Lemma \ref{lem:disjoint-splitting}, suppose that in addition that the matrix $A$ induces an $\N^{n_1+\cdots+n_k}$-grading on $R/I_A$. Then for all $i \geq 0$ and $\alpha \in \N^{n_1+\cdots+n_k}$,
  \[
    \beta_{i,\alpha}(R/I_A) = \sum_{\substack{i_1+\cdots+i_k = i \\ i_j \in \N}} \beta_{i_1,\alpha_1}(R/I_{A_1})\beta_{i_2,\alpha_2}(R/I_{A_2}) \cdots \beta_{i_k,\alpha_k}(R/I_{A_k})
  \]
  where \[\alpha_i = (\underbrace{0,\ldots,0}_{n_1+\cdots+n_{i-1}}, a_{i,1},\ldots,a_{i,n_i},\underbrace{0,\ldots,0}_{n_{i+1}+\cdots +n_k}) ~~\mbox{if $\alpha = (a_{1,1},\ldots,a_{k,n_k})$.}\]
\end{theorem}

\begin{proof}
  Let $R = \K[x_{i,1},\ldots,x_{k,s_k}]$ and set $R_i = \K[x_{i,1},\ldots,x_{i,s_i}]$. We give $R_i$ an $\N^{n_1+\cdots+n_k}$-grading by using the matrix $A_i$, but viewing $A_i$ as an $(n_1+\cdots+n_k) \times s_i$ matrix where the first $n_1+\cdots+n_{i-1}$ rows and the last $n_{i+1}+\cdots+n_k$ rows all consist of zeroes. As a consequence, if $\beta_{k,\delta}(R_i/I_{A_i}) \neq 0$, then ${\rm supp}(\delta) \subseteq \{n_1+\cdots+n_{i-1}+1,\ldots,n_1+\cdots+n_i\}$.

  If we abuse notation and view $I_{A_i}$ as both an ideal of $R$ and $R_i$, we have
  \[
    R/I_A = R/(I_{A_1}+\cdots+I_{A_k}) \cong R_1/I_{A_1} \otimes_\mathbb{K} R_2/I_{A_2} \otimes_\mathbb{K} \cdots \otimes_\mathbb{K} R_k/I_{A_k} \text{.}
  \]
  This follows since each $I_{A_i}$ generated by binomials only in the variables $\{x_{i,1},\ldots,x_{i,s_i}\}$. The multi-graded minimal free resolution of $R/I_A$ is then the tensor product of the multi-graded resolutions of the $R_i/I_{A_i}$'s (see \cite[Lemma 2.1]{JK} which does the standard graded case for $k=2$, but the proof extends naturally to the multi-graded case and to all $k$ by induction).

  It then follows by the K\"unneth formula that
  \[
    \beta_{i,\alpha}(R/I_A) = \sum_{\substack{i_1+\cdots+i_k = i \\ i_j \in \N}} \sum_{\substack{\gamma_1+\cdots+\gamma_k = \alpha \\ \gamma_j \in \N^{n_1+\cdots+n_k}}}\beta_{i_1,\gamma_1}(R_1/I_{A_1})\cdots \beta_{i_k,\gamma_k}(R_k/I_{A_k}) \text{.}
  \]
  As noted above, if ${\rm supp(\gamma)} \not\subseteq \{n_1+\cdots+n_{i-1}+1,\ldots,n_1+\cdots+n_i\}$, then $\beta_{k,\gamma}(R_i/I_{A_i}) = 0$. So
  we can assume the support of each index $\gamma_i$ is a subset of $\{n_1+\cdots+n_{i-1}+1,\ldots,n_1+\cdots+n_i\}$. But then the only decomposition $\gamma_1+\cdots+\gamma_k = \alpha$ that satisfies this condition is the decomposition $\alpha_1+\cdots+\alpha_k = \alpha$ with the $\alpha_i$'s defined as in the statement, and thus,
  \[
    \sum_{\substack{\gamma_1+\cdots+\gamma_k = \alpha \\ \gamma_j \in \N^{n_1+\cdots+n_k}}}\beta_{i_1,\gamma_1}(R_1/I_{A_1})\cdots \beta_{i_k,\gamma_k}(R_k/I_{A_k}) = \beta_{i_1,\alpha_1}(R_1/I_{A_1}) \cdots \beta_{i_k,\alpha_k}(R_k/I_{A_k}) \text{.}
  \]
  To complete the proof, note that $R_i/I_{A_i}$ and $R/I_A$ will have same graded Betti numbers with respect to our multi-grading, so we can replace each $\beta_{k,\gamma}(R_i/I_{A_i})$ with $\beta_{k,\gamma}(R/I_{A_i})$.
\end{proof}

\subsection{Toric ideals of graphs}
Let $G = (V(G),E(G))$ denote a finite simple
graph (a graph with no loops or multiple edges)
with vertex set $V(G) = \{x_1,\ldots,x_n\}$ and
edge set $E(G) = \{e_1,\ldots,e_q\}$
where each $e_i$ is a two-element
subset of $V$. Set $R = \mathbb{K}[E(G)] = \K[e_1,\ldots,e_q]$
and $S = \mathbb{K}[V(G)]= \K[x_1,\ldots,x_n]$, and define
the $\mathbb{K}$-algebra homomorphism
$\varphi\colon R \rightarrow S$ by 
\[e_i \mapsto x_{i,1}x_{i,2} ~~\mbox{where
$e_i = \{x_{i,1},x_{i,2}\}$}.\]
The {\em toric ideal} of $G$ is the ideal
$I_G = {\rm ker} \varphi$.

The toric ideal of $G$ is the toric
ideal of the incidence matrix of $G$.
More precisely, order the elements of $V(G)$ and 
$E(G)$, then the {\em incidence matrix} of $G$
is the $|V(G)| \times |E(G)|$ matrix $A$ where
$A_{i,j} = 1$ if $x_i \in e_j$ and $0$ otherwise.
Each column of $A$ contains only two
ones, and the remaining entries are zero.
Consequently, $I_G$ is both a graded ideal
(take the vector $c = \left(\frac{1}{2},\frac{1}{2},
\ldots,\frac{1}{2}\right)$ and apply Lemma
\ref{homogtoric})
and a multi-graded ideal. In particular, by \eqref{multi-standard-grading},
we have
\begin{equation}\label{toricidealgrading}
  \beta_{i,j}(I_G) = \sum_{|\alpha|=2j} \beta_{i,\alpha}(I_G).
\end{equation}

The dimension of $R/I_G$ depends upon whether
or not $G$ is bipartite. We say that $G$ is
a {\em bipartite graph} if there is a 
partition of the vertices $V(G) = V_1 \sqcup V_2$
such that every $e \in E(G)$ has the property
that $e \cap V_1 \neq \emptyset$ and $e \cap V_2
\neq \emptyset$. This is equivalent to having no odd cycles in $G$, a fact which we will make use of. Furthermore,
$G$ is {\em connected} if for
every $x,y \in V(G)$ with $x\neq y$, there
exists a sequence of edges $e_1,\ldots,e_t$ in
$E$ such that $x \in e_1$, $y \in e_t$,and
$e_{i} \cap e_{i+1} \neq \emptyset$
for $i=1,\ldots,t-1$.

\begin{theorem}\cite[Corollary 10.1.21]{V}
  \label{prop:bipartite-dim}
  If $G$ is a finite simple connected graph on $n$ vertices, then
  \[
    \dim(R/I_G) = \begin{cases}
      n &\ \text{if \textit{G} is not bipartite}\\
      n-1 &\ \text{if \textit{G} is bipartite.}\\
    \end{cases}
  \]
\end{theorem}

Work of Ohsugi-Hibi \cite{OH} and Villarreal \cite{VRees} 
allows us to describe the minimal generators of
$I_G$ in terms of the combinatorics of $G$. We
summarize the relevant results. 

\begin{definition}
  Let $G$ be a finite simple graph. A \emph{walk} is a sequence of edges $w=(e_{1},\dots,e_{k})$ such that $e_{i} \cap e_{i+1} \neq \emptyset$
  for $i=1,\ldots,k$. This is equivalent to specifying a sequence of vertices $(x_{1},\dots,x_{k},x_{k+1})$ such that $G$ has an edge which is associated to any consecutive $x_{i}$ and $x_{i+1}$ in the sequence. A walk is an \emph{even} walk of $k$ is even. A \emph{closed walk} is a walk which has a vertex sequence $(x_{1},\dots,x_{k+1})$ such that $x_{1}=x_{k+1}$.
\end{definition}

In the sequel, we will also require the family of path graphs.
The \emph{path graph} $P_{n}$ is the graph with vertex set $V(P_{n})=\{x_{1},\dots, x_{n+1}\}$ and edge set $E(P_{n})=\{\{x_{1},x_{2}\},\dots,\{x_{n},x_{n+1}\}\}$.

Closed even walks in $G$ correspond to elements of $I_G$. Indeed, let $w=(e_{i_{1}},e_{i_{2}},\dots, e_{i_{2n}})$ be a closed even walk corresponding to the following sequence of vertices $(x_{j_1}, x_{j_2}, \dots, x_{j_{n+1}} = x_{j_1})$ which are not necessarily distinct. We associate the walk $w$ with the binomial
\[
  f_w=\prod_{2 \nmid j}e_{i_j} - \prod_{2 \mid j} e_{i_j} \text{.}
\]
This element belongs to the ideal $I_{G}$ since 
\begin{align*}
  \phi(f_{w})&= \phi(e_{i_1})\phi(e_{i_3}) \cdots \phi(e_{i_{2n-1}}) - \phi(e_{i_2}) \phi(e_{i_4}) \cdots \phi(e_{i_{2n}})
  = \prod_{k=1}^n x_{j_k} - \prod_{k=1}^n x_{j_k} = 0 \text{.}
\end{align*}
The set of all binomials associated with closed even walks forms a generating set of $I_G$.  Using the following notion, we can further reduce our
generating set.

\begin{definition} Let $I$ be a toric ideal.
  A binomial $x^{\alpha_+} - x^{\alpha_-} \in I$ is \emph{primitive} if there exists no binomial $x^{\beta_+} - x^{\beta_-} \in I$ such that $x^{\beta_+} \mid x^{\alpha_+}$ and $x^{\beta_-} \mid x^{\alpha_-}$. A closed even walk $w$ in a graph $G$ is \emph{primitive} if the corresponding binomial $f_w$ is primitive in $I_G$.
\end{definition}

\begin{theorem}\cite[Proposition 10.1.10]{V}\label{thm:univ-grb}
  Let $G$ be a finite simple graph. Then $I_G$ is generated by binomials which correspond to closed even walks that are also primitive.
\end{theorem}
 
\begin{remark}
  The conclusion of \cite[Proposition 10.1.10]{V} is stronger where it
  is shown that the closed even walks that are primitive correspond to a universal Gr\"obner basis of $I_G$. \end{remark}

%%%%%%%%%%%%%%%%%%%%%%%%%%%%%%%%%%%%%%%%%%%%%%%%%%%%%%%%%%%%%%%%%%%%%%
\section{Splitting of toric ideals}
\label{sec:split-tor-id}

Given an $n \times s$ matrix $A$ with entries in $\Z$, we give a sufficient condition on $A$ that implies that the toric ideal $I_A$ is splittable, i.e., $I_A$ can be written as the sum of two (or more) toric ideals.
Although $I_A$ need not be (multi-)graded, when $A$ is chosen so that $I_A$ is also $\N^n$-graded, we can describe the multi-graded Betti numbers
in terms of those of the Betti numbers of the smaller ideals.
This result will be the consequence of the following technical lemmas.
\begin{lemma}
  \label{lem:tech1}
  Let $\alpha, \beta \in \Z^s$ be two linearly
  independent vectors with positive and negative entries such that $\gamma = \alpha + \beta$ also has positive and negative entries.
  Then $(x^{\alpha_+}-x^{\alpha_-})\nmid (x^{\gamma_+} - x^{\gamma_-})$ and
  $(x^{\beta_+}-x^{\beta_-})\nmid (x^{\gamma_+} - x^{\gamma_-})$.\end{lemma}
\begin{proof}
 We prove only the statement about
 $x^{\alpha_+}-x^{\alpha_-}$ since the other
 statement is similar.  Suppose that
  $x^{\gamma_+} - x^{\gamma_-} = f\cdot(x^{\alpha_+} - x^{\alpha_-})$.  If
  $f = f_1 + f_2 + \cdots+f_s$, where the $f_i$'s are the terms of $f$, then
  when we expand out the right hand side, we get
  \[f_1x^{\alpha_+} + f_2x^{\alpha_+} + \cdots + f_sx^{\alpha_+} -
  f_1x^{\alpha_-} - f_2x^{\alpha_-} - \cdots - f_sx^{\alpha_-}.
  \]
  If $f =f_1$ was a single monomial term, then we would
  have $f_1 = x^{\gamma_+- \alpha_+}$ and $f_1 = x^{\gamma_--\alpha_-}$, or in
  other words, $\gamma = \gamma_+-\gamma_- = \alpha_+-\alpha_- = \alpha$, and thus $\beta =
  0$, contradicting our choice of $\beta$.  So $s \geq 2$.
  Furthermore, the monomial $x^{\gamma_+}$ appears exactly once in the expansion.
  Indeed, if $x^{\gamma_+} = f_ix^{\alpha_+} - f_jx^{\alpha_-}$ for some $i \neq j$,
  this means that $\supp(\alpha_+) \cup \supp(\alpha_-) \subseteq \supp(\gamma_+)$.
  But the support of $\gamma_-$ is disjoint from that of $\gamma_+$.  However,
  the support of every term in the expansion contains $\supp(\alpha_+)$ or $\supp(\alpha_-)$,
  which means that $\gamma_-$ cannot appear on the right hand side.  The
  same argument now also applies to $x^{\gamma_-}$.

  So, without loss of generality,
  suppose that $f_1x^{\alpha_+} = x^{\gamma_+}$ and $f_sx^{\alpha_-} = 
  x^{\gamma_-}$ (note that we could have $f_1x^{\alpha_+} = x^{\gamma_-}$ and $f_sx^{\alpha_-} =  x^{\gamma_+}$, but our argument will also work for this case).
  So, $f_1 = x^{\gamma_+-\alpha_+}$.  The term $f_1x^{\alpha_-} = x^{\gamma_+-\alpha_++\alpha_-}$ must now cancel out with some term of the form
  $f_ix^{\alpha_+}$, say $f_2x^{\alpha_+} = f_1x^{\alpha_-}$ after relabelling.  
  This means that $f_2 = x^{\gamma_+ -2\alpha_++\alpha_-}$.  Now $f_2x^{\alpha_-}$ must
  cancel with some term $f_ix^{\alpha_+}$, say $f_3x^{\alpha_+}$.  This
  forces $f_3 = x^{\gamma_+ -3\alpha_++2\alpha_-}$.  Repeating this argument
  gives that $f_i = x^{\gamma_+ -i\alpha_++(i-1)\alpha_-}$ for $i=1,\ldots,s$.
  Since $f_sx^{\alpha_-} = x^{\gamma_-}$, we have $\gamma_+-s\alpha_++s\alpha_- =
  \gamma_-$.  Consequently, $\gamma = \alpha+\beta = s\alpha$, i.e.,
  $\beta = (s-1)\alpha$, contradicting our assumption on linearly independence.
\end{proof}

The next lemma can be viewed as giving a criterion
for ideal membership in a binomial ideal generated by two elements.
\begin{lemma}
  \label{lem:tech}
  Let $\alpha, \beta \in \Z^s$ be two linearly
  independent vectors with positive and negative entries such that $\gamma = \alpha + \beta$ also has positive and negative entries.
  Then $x^{\gamma_+} - x^{\gamma_-} \in \langle x^{\alpha_+} - x^{\alpha_-}, x^{\beta_+} - x^{\beta_-} \rangle$ if and only if $\supp(\alpha_+) \cap \supp(\beta_-) = \emptyset$ or $\supp(\alpha_-) \cap \supp(\beta_+) = \emptyset$.\end{lemma}
\begin{proof}
  We show the implication ``$\Rightarrow$'' by contradiction, i.e., we assume $x^{\gamma_+} - x^{\gamma_-}$ is contained in $\langle x^{\alpha_+} - x^{\alpha_-}, x^{\beta_+} - x^{\beta_-} \rangle$ and both $\supp(\alpha_+) \cap \supp(\beta_-) \neq \emptyset$ and $\supp(\alpha_-) \cap \supp(\beta_+) \neq \emptyset$.
  The binomial $x^{\gamma_+} - x^{\gamma_-}$ is contained in $\langle x^{\alpha_+} - x^{\alpha_-}, x^{\beta_+} - x^{\beta_-} \rangle$ if and only if
  \begin{equation}
    \label{eq:binom}
    x^{\gamma_+} - x^{\gamma_-} = f \cdot (x^{\alpha_+} - x^{\alpha_-}) + g \cdot (x^{\beta_+} - x^{\beta_-}) \text{,}
  \end{equation}
  for some non-zero polynomials $f, g \in \K[x_1, \ldots, x_s]$ by Lemma
  \ref{lem:tech1}. It follows that one of the monomials $x^{\alpha_+}$, $x^{\alpha_-}$, $x^{\beta_+}$, $x^{\beta_-}$ divides $x^{\gamma_+}$, respectively $x^{\gamma_-}$.
  
  Note that neither $x^{\alpha_+}$ nor $x^{\beta_+}$ divide $x^{\gamma_+}$. To see why, suppose $j \in \supp(\alpha_+) \cap \supp(\beta_-)$.
    Then $x_j^d$ appears in the monomial $x^{\alpha_+}$ and
    $x_j^e$ appears in the monomial $x^{\beta_-}$ for some integers
    $d,e \geq 1$.  The $j$-th coordinate of $\gamma$ is then $d-e$.
    If $d-e \geq 1$, then $x_j^{d-e}$ appears in the monomial
    $x^{\gamma_+}$, and so $x^{\alpha_+}$ cannot divide this monomial.
    If $d-e \leq 0$, then no power of $x_j$ appears in $x^{\gamma_+}$,
    and so again, $x^{\alpha_+}$ does not divide $x^{\gamma_+}$.  A similar
    argument holds for $x^{\beta_+}$. So, up to swapping $\alpha$ and $\beta$, we may assume that $x^{\beta_-}$ divides $x^{\gamma_+}$.

  Since $\supp(\gamma_+) \subseteq \supp(\alpha_+) \cup \supp(\beta_+)$, we obtain $\supp(\beta_-) \subseteq \supp(\alpha_+)$ and $\supp(\gamma_-) \subseteq \supp(\alpha_-)$ using the fact that $\supp(\alpha_+) \cap \supp(\alpha_-) = \emptyset$,
    and similarly for $\beta_+$ and $\beta_-$.
We conclude the preparatory observations by noting that neither $x^{\alpha_-}$, $x^{\beta_-}$, nor $x^{\alpha_+}$ divide $x^{\gamma_-}$, so that $x^{\beta+}$ must divide $x^{\gamma_-}$.
  To summarize, $\supp(\beta_-) \subseteq \supp(\gamma_+) \subseteq \supp(\alpha_+)$ and $\supp(\beta_+) \subseteq \supp(\gamma_-) \subseteq \supp(\alpha_-)$.  

We now claim that $\gamma_+ = \alpha_+ - \beta_-$ and $\gamma_- = \alpha_- - \beta_+$.
For the first equality, observe that there are three ways for $\gamma$ to have a positive value
in the $j$-th coordinate: $(1)$ the $j$-th coordinates of $\alpha$ and $\beta$ are both non-negative and at least one coordinate is positive, $(2)$ the $j$-th coordinate of $\alpha$, say $a_j$,  is positive,
and the $j$-th coordinate of $\beta$, say $b_j$, is negative, but $a_j + b_j > 0$, or $(3)$ the $j$-th coordinate of $\beta$, say $b_j$,  is positive, and the $j$-th coordinate of $\alpha$, say $a_j$, is negative, but $b_j + a_j > 0$.
  However, as $\supp(\beta_+) \subseteq \supp(\gamma_-)$ and $\supp(\beta_+) \cap \supp(\alpha_+) = \emptyset$, case $(1)$ can only happen if the $j$-th coordinate of $\alpha$ is positive and the one of $\beta$ vanishes.
  Case $(3)$ is impossible since this implies that $j \in \supp(\beta_+) \subseteq \supp(\gamma_-)$ and $j \in \supp(\gamma_+)$.
  This leaves case $(2)$, so that we can conclude $\gamma_+ = \alpha_+ - \beta_-$.  The second equality is proved similarly.

As $x^{\alpha_+}$ and $x^{\alpha_-}$ do not divide $x^{\gamma_+}$ and $x^{\gamma_-}$ respectively, we have $x^{\gamma_+}$ (resp.~$x^{\gamma_-}$) is a multiple of $x^{\beta_-}$ (resp.~$x^{\beta_+}$), i.e.,  $g = g' - x^{\gamma_+-\beta_-} - x^{\gamma_--\beta_+}$ for some $g' \in \K[x_1, \ldots, x_s]$, so that \eqref{eq:binom} becomes:
  \begin{equation}
    \label{eq:simplified}
    -f \cdot (x^{\alpha_+} - x^{\alpha_-}) = g' \cdot (x^{\beta_+}-x^{\beta_-}) + x^{\gamma_--\beta_+}\cdot x^{\beta_-}-x^{\gamma_+-\beta_-}\cdot x^{\beta_+} \text{.}
  \end{equation}
  Note that, $x^{\alpha_-} \nmid x^{\gamma_--\beta_++\beta_-}=x^{\gamma_--\beta}$.
  If $x^{\alpha_+} \nmid x^{\gamma_--\beta}$, then $x^{\gamma_--\beta}$ must be cancelled by a multiple of $x^{\beta_+}$, i.e., $g' = g''-x^{\gamma_--2\beta_++\beta_-}$ for some $g'' \in \K[x_1, \ldots, x_s]$.
  We obtain:
  \[
    -f \cdot (x^{\alpha_+} - x^{\alpha_-}) = g''\cdot (x^{\beta_+}-x^{\beta_-}) + x^{\gamma_--2\beta_++2\beta_-}-x^{\gamma_+-\beta_-}\cdot x^{\beta_+}\text{.}
  \]
  Again, $x^{\alpha_-} \nmid x^{\gamma_--2\beta_++2\beta_-}$, so that, if $x^{\alpha_+} \nmid x^{\gamma_--2\beta_++2\beta_-}$, we can repeat the same step again.
  This process must eventually stop, and we obtain that $x^{\alpha_+} \mid x^{\gamma_--k\beta_++k\beta_-}$ for some positive integer $k$.
  Then $k\beta_+ \le \gamma_- = \alpha_--\beta_+$ and $\alpha_+\le k\beta_-$ (where the inequalities are meant coordinate-wise).
  
  If we go back to equation \eqref{eq:simplified}, and repeat the same reasoning for the monomial $x^{\gamma_+-\beta_-+\beta_+}$, we obtain that $x^{\alpha_-} \mid x^{\gamma_+-\ell\beta_-+\ell \beta_+}$ for some positive integer $\ell$, and thus $\ell\beta_-\le \gamma_+=\alpha_+-\beta_-$ and $\alpha_-\le\ell\beta_+$.
  Summarizing, we obtain:
  \[
    (k+1)\beta_+\le \alpha_-\le \ell\beta_+ \qquad \text{and} \qquad (\ell+1)\beta_- \le \alpha_+\le k \beta_- \text{.}
  \]
  Hence $k+1 \le \ell$ and $\ell+1 \le k$. A contradiction.
      
  For the converse implication ``$\Leftarrow$'', assume that $\supp(\alpha_+) \cap \supp(\beta_-) = \emptyset$ (the other case works similarly). If $\delta =\beta_+ - \alpha_- \in \Z^s$, and thus, $\delta_++ \alpha_- = \delta_- + \beta_+$, then
  \[
    x^{\alpha_++\mathbf{\delta}_+} - x^{\beta_-+\delta_-} = x^{\delta_+} \cdot \left( x^{\alpha_+} - x^{\alpha_-} \right) + x^{\delta_-} \cdot \left( x^{\beta_+} - x^{\beta_-} \right) \in \langle x^{\alpha_+}-x^{\alpha_-}, x^{\beta_+} - x^{\beta_-} \rangle \text{.}
  \]
  It remains to show that the left side of this equation coincides with $x^{\gamma_+} - x^{\gamma_-}$. Note $\supp(\alpha_+ + \delta_+) \subseteq \supp(\alpha_+) \cup \supp(\beta_+)$ and $\supp(\beta_-+\delta_-) \subseteq \supp(\beta_-) \cup \supp(\alpha_-)$. As $\supp(\alpha_+) \cap \supp(\beta_-) = \emptyset$, the support of $\alpha_+$ is disjoint from $\supp(\beta_-) \cup \supp(\delta_-)$. From this it straightforwardly follows that the supports of $\alpha_++\delta_+$ and $\beta_-+\delta_-$ are disjoint. The statement follows by the observation that $\alpha_++\delta_+-(\beta_-+\delta_-) = \alpha+\beta = \gamma$.
\end{proof}

If in Lemma \ref{lem:tech} the ideal is replaced by its saturation with respect to the monomial $x_1 \cdots x_s$, then the assumption on the supports can be dropped.

\begin{example}
  Let $\alpha = \mathbf{e}_1 + \mathbf{e}_2 - \mathbf{e}_4 - \mathbf{e}_5$ and $\beta = \mathbf{e}_4 + \mathbf{e}_5 - \mathbf{e}_2 - \mathbf{e}_3$ in $\Z^5$ such that $\gamma = \alpha+\beta = \mathbf{e}_1 - \mathbf{e}_3$. Note that the assumption on the supports of Lemma \ref{lem:tech} is not satisfied and that $x^{\gamma_+} - x^{\gamma-} \not \in I \coloneqq \langle x^{\alpha_+} - x^{\alpha_-}, x^{\beta_+} - x^{\beta_-}\rangle$. However, $x^{\gamma_+} - x^{\gamma_-}$ is contained in the saturation $I : (x_1 \cdots x_5)^\infty$.
\end{example}

The next lemma gives us a criterion for when a toric ideal $I_A$ is splittable.

\begin{lemma}\label{lem:lattice-ideal}
  Let $A_1, \ldots, A_k$ be matrices with integer entries of dimensions $n_i \times s_i$ ($i=1,\ldots, k$) and let $c_1, \ldots, c_l \in \Z^N$ with $N \ge n_1 + \cdots + n_k$. Consider the matrix
  \[
    A = \begin{bsmallmatrix}
      A_1 &&& \smash{\rule[-9pt]{.5pt}{12pt}} & &\smash{\rule[-9pt]{.5pt}{12pt}} \\
      &\raisebox{0pt}{$\scalebox{.75}{$\ddots$}$} && \raisebox{0pt}{$\scriptstyle c_1$} & \raisebox{0pt}{$\scriptstyle \ldots$} & \raisebox{0pt}{$\scriptstyle c_l$} \\
      &&A_k&\smash{\rule[-7pt]{.5pt}{12pt}} & & \smash{\rule[-7pt]{.5pt}{12pt}} \\
      &0&& & &
    \end{bsmallmatrix}\in \Z^{N \times (s_1 + \cdots + s_k + l)}\text{.}
  \]
  Let $U_i$ be the set of indices of the columns in which $A_i$ is located in the matrix $A$. Suppose $\ker(A) = \ker(A_1) \oplus \ldots \oplus \ker(A_k) \oplus \Z \tau$ for some $\tau \in \Z^{s_1+\cdots+s_k+l}$. If for all $i \in \{1,\ldots,k\}$, the set $U_i$ is disjoint from either $\supp(\tau_+)$ or $\supp(\tau_-)$, then
  \[
    I_{A} = I_{A_1} + \cdots + I_{A_k} + \langle x^{\tau_+} - x^{\tau_-}\rangle \text{.}
  \]
\end{lemma}
\begin{proof}
  As $I_{A} = \langle x^{\gamma_+} - x^{\gamma_-} ~|~ \gamma \in \ker(A) \rangle$,
  the inclusion ``$\supseteq$'' is clear. To prove the
  reverse inclusion, let $\gamma = \beta^1 + \cdots + \beta^k + c\tau
  \in \ker(A)$ where $\beta^i \in \ker(A_i)$ and
  $c \in \Z$.
  
  We do induction on $k$. The base case $k=0$ is straightforward.
  If $k>0$, then we set
  $\beta \coloneqq \beta^1+ \ldots + \beta^{k-1} + c\tau$. Note that our
  assumptions imply that $\supp(\beta_+)$ or $\supp(\beta_-)$ is disjoint
  from $\supp(\beta^k)$. By Lemma \ref{lem:tech}, $x^{\gamma_+} - x^{\gamma_-} \in \langle x^{\beta_+}-x^{\beta_-},
  x^{\beta^k_+} - x^{\beta^k_-}\rangle$ and we conclude the proof by the
  induction hypothesis, that is, $x^{\beta_+} - x^{\beta_-} \in
  I_{A_1} + \cdots + I_{A_{k-1}} + \langle x^{\tau_+} - x^{\tau_-} \rangle$.
\end{proof}

Note that in order to apply Lemma \ref{lem:lattice-ideal}, it might be necessary to choose a suitable basis, so that the matrix $A\in \mathbb{Z}^{n \times s}$ has the appropriate shape. However, when we restrict to toric ideals of graphs, Lemma \ref{lem:lattice-ideal} holds for some graph constructions (see Theorem \ref{t. Splitting 1}).

When a matrix $A$ that satisfies conditions
of Lemma \ref{lem:lattice-ideal} also induces a multi-grading, Lemma \ref{lem:lattice-ideal} has implications for the multi-graded Betti numbers.

\begin{theorem}\label{bettiformula}
  With the notation and hypotheses of Lemma \ref{lem:lattice-ideal}, suppose that in addition the matrix $A$ induces an $\N^N$-grading on $R/I_A$. Then for all $i \geq 0$ and $\alpha \in \N^N$,
  \[
    \beta_{i,\alpha}(R/I_A) = \beta_{i,\alpha}(R/J) + \beta_{i-1,\alpha-\mu}(R/J)
  \]
  where $J = I_{A_1}+\cdots +I_{A_k}$ and $\mu = \deg{(x^{\tau_+}-x^{\tau_-})} \in \N^N$.
\end{theorem}
\begin{proof}
  By Lemma \ref{lem:lattice-ideal}, we have $I_A = J + \langle x^{\tau_+} - x^{\tau_-}\rangle$, and furthermore, this ideal is $\N^s$-graded. Set $F = x^{\tau_+} - x^{\tau_-}$. We then have a multi-graded short exact sequence of $R$-modules
  \[
    0 \longrightarrow \left(R/(J:\langle F\rangle) \right)(-\mu) \xrightarrow{\times F} R/J \longrightarrow R/(J+\langle F \rangle) = R/I_A \longrightarrow 0 \text{.}
  \]
  The ideal $J$ is a toric ideal by Lemma \ref{lem:disjoint-splitting}, and consequently, it is prime. Since $F \not\in J$, it then follows that $J:\langle F \rangle = J$. So we can rewrite the short exact sequence above as
  \begin{equation}
    \label{e:ses}
    0 \longrightarrow (R/J)(-\mu) \xrightarrow{\times F} R/J \longrightarrow R/I_A \longrightarrow 0 \text{.}
  \end{equation}
  Let $(\mathcal{H},d)$ denote the multi-graded minimal free resolution of $R/J$. Then the multi-graded minimal free resolution $(\mathcal{G},d')$ of $(R/J)(-\mu)$ is the same except all the free $R$-modules in $\mathcal{H}$ will have their grading twisted by $\mu$. Hence the map $\times F\colon (R/J)(-\mu) \to R/J$ lifts to a map of complexes $\phi\colon (\mathcal{G},d') \to (\mathcal{H},d)$ where $\phi_i\colon G_i \to H_i$ is the map $\phi_i$ that takes each basis element of $G_i$ and multiplies it by $F$.

  The mapping cone construction applied to \eqref{e:ses}, gives a minimal
  multi-graded free resolution of $R/I_A$. Indeed, the resolution produced by the mapping cone construction is minimal if all maps $\phi_i$ can be represented by matrices where none of the non-zero entries of the matrices are constants. In our case, all the non-zero entries are $F$. The multi-graded Betti numbers in the statement now follow from our minimal multi-graded free resolution.
\end{proof}

\begin{remark}
  The multi-graded Betti numbers of $R/J$ can be computed by
  Theorem \ref{kunneth}. Hence, under the hypotheses Theorem \ref{bettiformula},
  the multi-graded Betti numbers of $I_A$ only depend on
  the Betti numbers of the ideals in the splitting of $I_A$.
\end{remark}

If we specialize our results to toric ideals of graphs,
Lemma \ref{lem:lattice-ideal} allows us to find splittings of $I_G$ in
terms of graph theoretic constructions. In particular, the technical
hypotheses of Lemma \ref{lem:lattice-ideal} correspond to a graph theoretic construction of taking a large even cycle, and joining (mostly bipartite) graphs in a prescribed manner to this cycle.

\begin{theorem}\label{t. Splitting 1}
  Let $G_1, \ldots, G_k$ be finite simple connected graphs with at most one $G_i$ not being bipartite. 
  Let $C$ be an even cycle with at least $k$ edges.
  For each $i$, identify an edge of $G_i$ 
  with a distinct edge of $C$ (see Figure \ref{fig:combined}). Then the toric ideal $I$ of the resulting graph is given by
  \[
    I = I_{G_1} + \cdots + I_{G_k} + \langle f \rangle\text{,}
  \]
  where $I_{G_i}$ is the toric ideal of $G_i$ and $f$ is the binomial corresponding to the even cycle $C$.
\end{theorem}
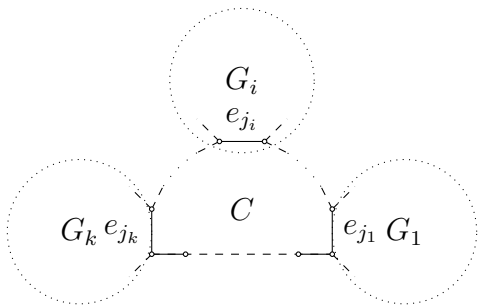
\begin{figure}[!ht]
  \centering
  \begin{tikzpicture}[scale=0.3]
    \fill[white] (0,6) circle (3cm);
    \fill[white] (-6,0) circle (3cm);
    \fill[white] (6,0) circle (3cm);

    \draw[dotted] (0,6.7) circle (3.2cm) node{$G_i$};
    \draw[dotted] (-7.2,0) circle (3.2cm) node{$G_k$};
    \draw[dotted] (7.2,0) circle (3.2cm) node{$G_1$};
    
    \draw (-4,1) -- (-4,-1) node[midway,left] {$e_{j_k}$};
    \draw (4,1) -- (4,-1) node[midway, right] {$e_{j_1}$};
    \draw (-1,4) -- (1,4) node[midway,above] {$e_{j_i}$};
    \draw (-4,-1) -- (-2.5,-1);
    \draw (2.5,-1) -- (4,-1);
    \draw[loosely dotted] (3.5,2) edge[bend right=10] (2,3.5);
    \draw[dashed] (4,1) -- (3.5,2);
    \draw[dashed] (1,4) -- (2,3.5);
    \draw[dashed] (-1,4) -- (-2,3.5);
    \draw[dashed] (-4,1) -- (-3.5,2);
    \draw[loosely dotted] (-2,3.5) edge[bend right=10] (-3.5,2);
    \draw[dashed] (-4,-1) -- (4,-1);
    \draw[thin,dashed] (-4,1) -- (-5,2);
    \draw[thin,dashed] (-4,-1) -- (-5,-2);
    \draw[thin,dashed] (4,1) -- (5,2);
    \draw[thin,dashed] (4,-1) -- (5,-2);
    \draw[thin,dashed] (-1,4) -- (-2,5);
    \draw[thin,dashed] (1,4) -- (2,5);
    \node at (0,1) {$C$};
    \fill[fill=white,draw=black] (-4,1) circle (.1);
    \fill[fill=white,draw=black] (-4,-1) circle (.1);
    \fill[fill=white,draw=black] (-2.5,-1) circle (.1);
    \fill[fill=white,draw=black] (4,1) circle (.1);
    \fill[fill=white,draw=black] (4,-1) circle (.1);
    \fill[fill=white,draw=black] (2.5,-1) circle (.1);
    \fill[fill=white,draw=black] (-1,4) circle (.1);
    \fill[fill=white,draw=black] (1,4) circle (.1);
  \end{tikzpicture}
  \caption{Connecting graphs $G_1, \ldots, G_k$ along $k$ edges of an even    cycle $C$.}
  \label{fig:combined}
\end{figure}
\begin{proof}
  Let $A_i$ be the incidence matrix of $G_i$, i.e., $A_i$ is an $n_i \times s_i$ matrix where $n_i=|V(G_{i})|$ and $s_i=|E(G_{i})|$. Note that $\rank(A_i) \in \{ n_i, n_i - 1 \}$ with at most one matrix having rank $n_i$ (if $G_i$ is not bipartite) by Theorem \ref{prop:bipartite-dim}. Let $l$ be the number of additional edges, so that the resulting graph has $n_1 + \cdots + n_k + l-k$ vertices and $s \coloneqq s_1 + \cdots +s_k+l$ edges. Let $B$ be the $(n_1+\cdots+n_k+l-k)\times(s_1+\cdots+s_k+l)$ incidence matrix of the resulting graph $G$ whose shape coincides with the shape of the matrix in Lemma \ref{lem:lattice-ideal} where the block-diagonal part is built from the matrices $A_i$ and the additional $l$ columns correspond to the additional edges. It is straightforward to verify that the even cycle $C$ induces an element $\tau$ in the kernel of $B$, so that $L \coloneqq \ker(A_1) \oplus \ldots \oplus \ker(A_k) \oplus \Z \tau \subseteq \ker(B)$.

  Next, we determine the rank of $\ker(B)$. We distinguish two cases: If all $G_i$ are bipartite, then $G$ is also bipartite, and thus by Theorem \ref{prop:bipartite-dim}
  \begin{align*}
    \rank(\ker(B)) &= s - \rank(B) = s_1 + \cdots + s_k + l - (n_1 + \cdots + n_k + l - k - 1) \\
    &= \left(s_1 - \left(n_1 - 1\right)\right) + \cdots + \left(s_k - \left(n_k - 1\right)\right) + 1 \text{.}
  \end{align*}
  Similarly, if say $G_1$ is not bipartite, then $G$ is not bipartite, and thus by Theorem \ref{prop:bipartite-dim}
  \begin{align*}
    \rank(\ker(B)) &= s - \rank(B) = s_1 + \cdots + s_k + l - (n_1 + \cdots + n_k+l-k)\\
    &= \left(s_1 - n_1\right) + \left( s_2 - \left(n_2-1\right)\right) + \cdots + \left(s_k-\left(n_k-1\right)\right) + 1 \text{.}
  \end{align*}
  We conclude that in either case $\rank(\ker(B)) = \rank(\ker(A_1)) + \ldots + \rank(\ker(A_k)) + 1$. However, to show the equality $L = \ker(B)$, it remains to show that $L$ is saturated in $\ker(B)$. If $\alpha \in \ker(B)$ such that $k \alpha \in L$ for some integer $k$, then $k\alpha = \beta + u \tau$ for some $\beta \in \ker(A_1) \oplus \ldots \oplus \ker(A_k)\eqqcolon L'$ and some integer $u$. As $L' \subseteq \Z^{s_1+\ldots+s_k} \times \{0\}^l$ and $\tau$ has an entry ``$\pm1$'' in its last $l$ coordinates, we can conclude that $k$ divides $u$, say $u = k u'$, so that $\beta = k(\alpha - u' \tau)$. As $L'$ is the kernel of the matrix obtained from $B$ by replacing the last $l$ columns by $0$-columns, it follows that $L'$ is saturated in $\Z^{s_1+\ldots+s_k+l}$, and thus $\beta' \coloneqq \alpha - u' \tau \in L'$. Hence $\alpha = \beta' + u' \tau$ is contained in $L$ which concludes the proof that $L$ is saturated in $\ker(B)$, and therefore the two lattices coincide. If $U_i$ is as in Lemma \ref{lem:lattice-ideal}, then, since $|\supp(\mathbf{\tau}) \cap U_i| = 1$, the result follows by Lemma \ref{lem:lattice-ideal}.
\end{proof}

\begin{remark}
  Note that Theorem \ref{t. Splitting 1} is independent of the edge we pick in each $G_i$. 
\end{remark}

We end this section by recording some consequences for the graphs of Theorem \ref{t. Splitting 1}.

\begin{theorem}\label{gradedbettinumbersgraphs}
  Let $G_1, \ldots, G_k$ be finite simple connected graphs with at most one $G_i$ not being bipartite. Let $G$ be the graph constructed as in Theorem \ref{t. Splitting 1}. If the even cycle $C$ has size $2d$, then
  \[
    \beta_{i,j}(R/I_G) = \beta_{i,j}(R/J) + \beta_{i-1,j-d}(R/J) ~~\mbox{for
    all $i,j \geq 0$}
  \]
  where $J = I_{G_1}+\cdots+I_{G_k}$ with each $I_{G_i}$ viewed as an ideal of $R$.
\end{theorem}
\begin{proof}
  The standard grading of $R/I_G$ is compatible with the multi-grading of $R/I_G$ given by the incidence matrix of $G$. Now combine Theorem \ref{bettiformula} with equation \eqref{toricidealgrading}, after  noting that the generator $f$ of Theorem \ref{t. Splitting 1} has $\deg(f) =d$ (in the standard grading).
\end{proof}

\begin{remark}
  By applying Theorem \ref{kunneth}, we also have a formula for the graded Betti numbers of $R/J$ in Theorem \ref{gradedbettinumbersgraphs}. In particular, if $J = I_{G_1}+\cdots +I_{G_k}$, we have
  \[
    \beta_{i,j}(R/J) = \sum_{\substack{i_1 + \cdots + i_k = i \\ i \geq 0}} \sum_{\substack{j_1 + \cdots + j_k = j \\ j \geq 0}} \beta_{i_1, j_1}(R/I_{G_1}) \cdots \beta_{i_k, j_k}(R/I_{G_k}) \text{.}
  \]
\end{remark}

We record some consequences for the homological invariants. Let $I$ be a homogeneous ideal in the standard graded polynomial ring $S = \K[x_1,\ldots,x_n]$. The {\em Hilbert series} of a standard graded $\mathbb{K}$-algebra $S/I$ is the formal power series
\[
  HS_{S/I}(t) = \sum_{i \geq 0} \left[\dim_\K (S/I)_i\right]t^i.
\]
By the Hilbert-Serre Theorem (e.g., \cite[Theorem 5.1.4]{V}) there is an $h_{S/I}(t) \in \Z[t]$ such that
\[
  HS_{S/I}(t) = \frac{h_{S/I}(t)}{(1-t)^{\dim(R/I)}} ~~\mbox{with $h_{S/I}(1) \neq 0$} \text{.}
\]
The polynomial $h_{S/I}(t)$ is the {\em $h$-polynomial} of $S/I$. The {\em (Castelnuovo-Mumford) regularity} is
\[
  \reg(S/I) = \max\{j-i ~|~ \beta_{i,j}(S/I) \neq 0\} \text{.}
\]
The {\em projective dimension} of $S/I$ is the length of the graded minimal free resolution, that is
\[
  \pdim(S/I) = \max\{i ~|~ \beta_{i,j}(S/I) \neq 0\} \text{.}
\]
We now have:
\begin{corollary}
  Let $G_1, \ldots, G_k$ be finite simple connected graphs with at most one $G_i$ not being bipartite. Let $G$ be the graph constructed as in Theorem \ref{t. Splitting 1}. If the even cycle $C$ has size $2d$, then
  \begin{enumerate}[label=(\roman*)]
  \item $h_{R/I_G}(t) = \frac{(1-t^d)}{(1-t)} \cdot \prod_{i=1}^k h_{R_i/I_{G_i}}(t)$ where $R_i = \K[E(G_{i})]$;
  \item $\reg(R/I_G) = \reg(R/I_{G_1}) + \cdots + \reg(R/I_{G_k}) + (d-1)$;
  \item $\pdim(R/I_G) = \pdim(R/I_{G_1}) + \cdots + \pdim(R/I_{G_k}) + 1$.
  \end{enumerate}
\end{corollary}
\begin{proof}
  Set $R_i = \K[E(G_{i})]$. Let $J = I_{G_1}+\cdots+I_{G_k}$, where we view each $I_{G_i}$ as an ideal of $S = 
  \K[E(G_{1})\cup \cdots \cup E(G_k)]$. Then
  \[
    S/J \cong R_1/I_{G_1} \otimes_\K R_2/I_{G_2} \otimes_\mathbb{K} \cdots \otimes_\K R_k/I_{G_k} \text{.}
  \]
  By tensoring the resolutions of each $R_i/I_{G_i}$ to construct a resolution of $S/J$ we get:
  \begin{multline*}
    h_{S/J}(t) = \prod_{i=1}^k h_{R_i/I_{G_i}}(t), ~~\reg(S/J) = \sum_{i=1}^k \reg(R_i/I_{G_i}), \;\mbox{and}\\
    \pdim(S/J) = \sum_{i=1}^k \pdim(R_i/I_{G_i}) \text{.}
  \end{multline*}
  
  Now view $J$ as an ideal of $R = \K[E(G)]$.  That is, $R$ is obtained by adjoining the variables to
  $S$ that correspond to the edges
  of $C$ that do not appear in $G_1,\ldots,G_k$.
  As shown in the proof of Theorem \ref{bettiformula}, we have a short exact sequence
  \[
    0 \to (R/J)(-d) \xrightarrow{\times f} R/J \to R/I_G \to 0 \text{,}
  \]
  where $f$ is the degree $d$
  binomial that 
  corresponds to the even cycle
  $C$.
  For statement $(i)$, the Hilbert series are additive on short exact sequences. So $HS_{R/I_{G}}(t) = HS_{R/J}(t)-t^dHS_{R/J}(t) = (1-t^d)HS_{R/J}(t)$.
  Since $h_{R/J}(t) = h_{S/J}(t)$, the numerator of the reduced
  Hilbert series for $R/I_G$ is $\frac{(1 - t^d)}{(1-t)} \cdot \prod_{i=1}^k h_{R_i/I_{G_i}}(t)$.
  Statements $(ii)$ and $(iii)$ are consequences of the fact that the mapping cone construction on this short exact sequence produces a minimal graded free resolution, and the fact that the regularity and the projective
  dimension of $I_{G_i}$ and $I_G$ remain the same when we view them as ideals
  in the ring $R$.
\end{proof}

We end with a special case of Theorem \ref{gradedbettinumbersgraphs} which justifies the 
example in the introduction.

\begin{corollary}\label{c. firstglue}
  Let $G$ be any finite simple graph. Fix an edge $e$ in $G$, and connect a new even cycle of length $2d \geq 4$ along $e$ (see Figure \ref{fig:onecycle}). If $H$ is the resulting graph, then
  \[
    \beta_{i,j}(R/I_H) = \beta_{i,j}(R/I_G) + \beta_{i,j-d}(R/I_G)\quad\mbox{for all} \; i,j \geq 0 \text{.}
  \]
\end{corollary}

%%%%%%%%%%%%%%%%%%%%%%%%%%%%%%%%%%%%%%%%%%%%%%%%%%%%%%%%%%%%%%%%%

\section{Other splittings for toric ideals of graphs}
\label{s. Splittings of Toric Ideals of Graphs}

In this section we give another splitting of a 
toric ideal of a graph.  The starting point of 
our approach is the
observation that  Corollary \ref{c. firstglue} implies that if we ``glue" an even 
cycle onto the edge of a graph to make a new graph $G$, then
$I_G$ is the sum of the toric ideals
of ``glued" graphs, that is, $I_G$ is splittable.  The notion of
a ``gluing" also appears in \cite[Proposition 7.49 and Theorem 7.50]{Binomi}
where the authors show how some properties of the toric 
ideals of graphs are preserved for a certain class of graphs after ``gluing" the graphs at one vertex. 

We formalize the notion of gluing, and a corresponding inverse
operation, which we call a splitting. Note that 
variations of this construction have appeared in the literature (e.g., Koh and Teo \cite{KT} describes a gluing along
a complete graph); other examples undoubtedly exist.
For our constructions we require
induced subgraphs.  Given a graph $G = (V(G),E(G))$ and
$W \subseteq V(G)$, the \emph{induced subgraph} of $G$
on $W$ is the graph $H$ with $V(H)=W$ and $E(H)=\{e \in E(G) ~|~ e \subseteq W\}$.

\begin{construction}
  Let $G_1, G_2$ be two graphs and suppose that
  $H_1 \subseteq G_1, H_2\subseteq G_2$ are two induced subgraphs
  which are isomorphic with respect to some graph isomorphism
  $\varphi \colon H_1 \to H_2$. We define the \emph{glued graph}
  $G_1 \cup_\varphi G_2$ of $G_1$ and $G_2$ along $\varphi$ as the
  disjoint union of $G_1$ and $G_2$, and then using $\varphi$ to
  identify associated vertices and edges.  At times, we may be more
  informal and say that  \emph{$G_1$ and $G_2$ is glued along $H$} if
  the induced subgraphs $H \cong H_1$ and $H \cong H_2$ and
  isomorphism $\varphi$ are clear.
\end{construction}
\begin{construction}
  Let $G = (V(G), E(G))$ be a finite simple graph. Suppose there are
  two subsets $W_1,W_2 \subseteq V(G)$ whose union gives $V(G)$, and
  denote the induced subgraph with vertex set $W_i$ by $G_i$ for $i = 1, 2$.
  Let $Y = W_1 \cap W_2$ and denote the corresponding induced subgraph by
  $H$. We say that \emph{$G_1$ and $G_2$ form a splitting of $G$ along $H$} if
  the graph obtained by removing the vertices $Y$ from $G$ yields two
  disconnected pieces.
\end{construction}

The two constructions given above are inverses of each other in
the following sense.  If $G$ is the glued graph of $G_1$ and $G_2$
along $\varphi$, then $G_1$ and $G_2$ form a splitting of $G$ along $H$
where we identify $G_i$ with the corresponding induced subgraph in $G$ and
where $H$ is the induced subgraph of $G$ corresponding to $H_i$. Inversely,
if $G$ is a finite graph and $G_1, G_2$ are two induced subgraphs which
form a splitting of $G$ along some common induced subgraph $H \subseteq G_i$,
then $G$ can be obtained from $G_1$ and $G_2$ as the corresponding glued graph.

\begin{remark} Using the analogy of direct products of groups, note that
  a gluing of graphs is similar to an external direct products of groups in
  the sense that the glued graph is constructed from two given graphs.  On
  the other hand, we can view a splitting of a graph as similar to an
  internal direct product in that we are decomposing the graph in terms
  of subgraphs.  Depending upon the context, one point-of-view may be
  preferable.
  \end{remark}

Different choices of the isomorphism $\varphi$ can result in non-isomorphic glued graphs.
 
\begin{example}\label{e.gluing} 
  Let $G_1 = G_2$ be the graph in Figure \ref{f. gluing edges1}. Consider the edge $H_1 = H_2 = \{ x_1, x_2 \}$.
  \begin{figure}[!ht]
    \centering
    \hfill
    \begin{tikzpicture}[scale=.36]
      \node (f1) at (-2.9,1.5) {$G_1 = G_2 = $};
      \node[xshift=1pt] (f2) at (1.5,-1.5) {\small $H_1=H_2$};
      % define points
      \coordinate (v1) at (0,0);
      \coordinate (v2) at (3,0);
      \coordinate (v3) at (3,3);
      \coordinate (v4) at (0,3);
      \coordinate (v5) at (5,1.5);

      % draw edges
      \draw [very thick] (v1)--(v2); 
      \draw (v2)--(v3);
      \draw (v3)--(v4);
      \draw (v4)--(v1);
      \draw (v2)--(v5);
      \draw (v3)--(v5);

      % draw graph nodes and their labels	
      \fill[fill=white,draw =black] (v1) circle (.1) node [below left] {$x_1$};
      \fill[fill=white,draw =black] (v2) circle (.1) node [below right]{$x_2$};
      \fill[fill=white,draw =black] (v3) circle (.1);
      \fill[fill=white,draw =black] (v4) circle (.1);
      \fill[fill=white,draw =black] (v5) circle (.1);
    \end{tikzpicture}\hfill\hfill
    \caption{The graph from Example \ref{e.gluing}.}
    \label{f. gluing edges1} 
  \end{figure}
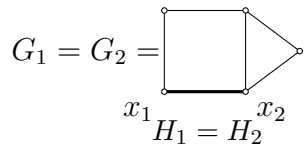
  The two possible choices of isomorphisms $\varphi \colon H_1 \to H_2$
  (depending on whether we flip the edge or not) yield non-isomorphic
  glued graphs (see Figure \ref{f. gluing edges2}). Indeed, one graph
  contains a vertex of degree five while the degree of any vertex in
  the other graph is at most four.
  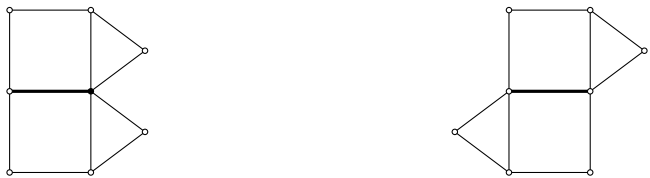
\begin{figure}[!ht]
    \centering
    \hfill
    \begin{tikzpicture}[scale=0.36]
      % define points
      \coordinate (v1) at (0,0);
      \coordinate (v2) at (3,0);
      \coordinate (v3) at (3,3);
      \coordinate (v4) at (0,3);
      \coordinate (v5) at (5,1.5);
      \coordinate (v6) at (0,-3);
      \coordinate (v7) at (3,-3);
      \coordinate (v8) at (5,-1.5);
	
      % draw edges
      \draw[very thick] (v1)--(v2);
      \draw (v2)--(v3);
      \draw (v3)--(v4);
      \draw (v4)--(v1);
      \draw (v2)--(v5);
      \draw (v3)--(v5);
      \draw (v1)--(v6);
      \draw (v2)--(v7);
      \draw (v6)--(v7);
      \draw (v8)--(v2);
      \draw (v8)--(v7);

      % draw graph nodes and their labels	
      \fill[fill=white,draw =black] (v1) circle (0.1);
      \fill[fill=black,draw =black] (v2) circle (0.1);
      \fill[fill=white,draw =black] (v3) circle (0.1);
      \fill[fill=white,draw =black] (v4) circle (0.1);
      \fill[fill=white,draw =black] (v5) circle (0.1);
      \fill[fill=white,draw =black] (v6) circle (0.1);
      \fill[fill=white,draw =black] (v7) circle (0.1);
      \fill[fill=white,draw =black] (v8) circle (0.1);
    \end{tikzpicture}\hfill
    \begin{tikzpicture}[scale=0.36]
      % define points
      \coordinate (v1) at (0,0);
      \coordinate (v2) at (3,0);
      \coordinate (v3) at (3,3);
      \coordinate (v4) at (0,3);
      \coordinate (v5) at (5,1.5);
      \coordinate (v6) at (0,-3);
      \coordinate (v7) at (3,-3);
      \coordinate (v8) at (-2,-1.5);

      % draw edges
      \draw[very thick] (v1)--(v2);
      \draw (v2)--(v3);
      \draw (v3)--(v4);
      \draw (v4)--(v1);
      \draw (v2)--(v5);
      \draw (v3)--(v5);
      \draw (v1)--(v6);
      \draw (v2)--(v7);
      \draw (v6)--(v7);
      \draw (v8)--(v1);
      \draw (v8)--(v6);	

      % draw graph nodes and their labels	
      \fill[fill=white,draw =black] (v1) circle (0.1);
      \fill[fill=white,draw =black] (v2) circle (0.1);
      \fill[fill=white,draw =black] (v3) circle (0.1);
      \fill[fill=white,draw =black] (v4) circle (0.1);
      \fill[fill=white,draw =black] (v5) circle (0.1);
      \fill[fill=white,draw =black] (v6) circle (0.1);
      \fill[fill=white,draw =black] (v7) circle (0.1);
      \fill[fill=white,draw =black] (v8) circle (0.1);
    \end{tikzpicture}\hfill\hfill
    \caption{Different ways to glue graphs along one edge.}
    \label{f. gluing edges2}
  \end{figure}
\end{example}

Although the gluing of $G_1$ and $G_2$ depends
upon the isomorphism $\varphi$, in some cases the toric
ideal of the glued graph is independent of $\varphi$.
Specifically, if at least one graph
is bipartite, and if we glue along a particular type of subgraph,
then the toric ideal of the glued graph is almost splittable (i.e.,
splittable up to a saturation with respect to a particular element).

\begin{theorem}\label{thm:glue-Gs-along-path}
  Let $G_1$ and $G_2$ be a splitting of a graph $G$ along a path graph $P_l
  \cong H \subseteq G$ such that any vertex of $H$ distinct from
  the endpoints considered as a vertex inside $G$ has degree $2$.
  If $G_1$ is bipartite, then we obtain.
  \[
    I_G = \left( I_{G_1} + I_{G_2} \right) : f^\infty \text{,}
  \]
  where $f$ denotes the square-free monomial corresponding
  to the edges in $H$ with even indices.
\end{theorem}
\begin{proof}
  The inclusion ``$\supseteq$'' follows by the fact that $I_{G_i}$
  is contained in $I_G$, and that $I_G$ is a prime ideal.

  For the reverse inclusion, recall from Theorem \ref{thm:univ-grb} that $I_G$ is generated by binomials corresponding to primitive closed even walks $p$ in $G$. Note that $p$ cannot contain a subpath in $G_1$ starting and ending at the same endpoint of $H$ (otherwise, as $G_1$ is bipartite, this subpath would be even, and thus $p$ is a concatenation of closed even walks contradicting the fact that $p$ was chosen to be primitive). Let us label the edges in $H$ by $h_1, \ldots, h_l$ and the remaining edges in $G_2$ by $h_{l+1}, \ldots, h_n$. Furthermore, label the edges in $G_1$ which are not contained in $H$ by $e_1, \ldots, e_m$. Using this notation, we can write a primitive closed even walk $p$ as follows
  \begin{multline}
    \label{eq:subpaths_G1}
    p = (\underbrace{e_{i_{11}}, \ldots, e_{i_{1r_1}}}_{\coloneqq p_1}, h_{j_{11}}, \ldots, h_{j_{1s_1}}, \underbrace{e_{i_{21}}, \ldots, e_{i_{2r_2}}}_{\coloneqq p_2}, h_{j_{21}}, \ldots, h_{j_{2s_2}}, \ldots,\\
    \underbrace{e_{i_{u1}}, \ldots, e_{i_{ur_u}}}_{\coloneqq p_u}, h_{j_{u1}}, \ldots, h_{j_{us_u}}) \text{.}
  \end{multline}
  We obtain subpaths $p_1, p_2, \ldots, p_u$ that contain edges
  of $E(G_1)\setminus E(G_2)$ that begin at one of the endpoints of $H$, and
  end at the other endpoint.
  
  We conclude the proof by showing that a path $p$ in $G$ that accepts a
  representation as in Equation \ref{eq:subpaths_G1} yields a binomial $f_p$
  contained in $(I_{G_1} + I_{G_2}) : f^\infty$. The proof is done by induction
  on the number of subpaths $p_1, p_2, \ldots, p_u$. If there are no such
  paths, then $p$ is contained entirely in $G_1$ or $G_2$, and the
  corresponding binomial belongs to the respective binomial ideal.

  If there is at least one such path $p_1$, we proceed as follows.
  To simplify notation, we write $p_1 = ( e_1, \ldots, e_r)$
  (here $r = r_1$) and $p = (e_1, \dots, e_{2m})$, where $e_{r+1},\ldots,
  e_{2m}$ is an edge in either $G_1$ or $G_2$ ($p$ contains an even
  number of edges since it is a primitive even walk).
  Furthermore, we denote the edges of the path graph
  $H$ by $(h_1, \ldots, h_l)$ (ordered such that they form a path starting at the endpoint of $p_1$). Our goal is to decompose the binomial $f_{p}$ into a linear combination of binomials $g_1$ and $f_{p'}$ corresponding to the closed even walks $(e_1, \dots, e_r, h_1, \dots, h_l)$ and $p' \coloneqq
  (e_{r+1}, \dots, e_{2n}, h_l,\dots, h_1)$ respectively. We define
  \begin{align*}
    E_1 &=\prod_{1\le k \le r, 2\mid k}e_k
    &O_{1}&=\prod_{1\le k\le r, 2\nmid k}e_k
    &F_1&=\prod_{2 \nmid k}h_{k}\\
    E_{2}&=\prod_{r+1\le k \le 2m, 2 \mid k}e_{p_{k}}
    &O_{2}&=\prod_{r+1\le k\le 2m, 2\nmid k}e_k
    &F_2&=\prod_{2\mid k}h_{k}.
  \end{align*}
  This allows us to write $f_p = O_1O_2 - E_1E_2$.

  If $l$ is even, we have $g_1 = O_1 F_1 - E_1 F_2 \in I_{G_{1}}$ and $f_{p'} = O_2 F_2 - E_2 F_1$.
  Note that since $l$ is even, then either $f = F_1$ or $f=F_2$ (i.e.,
  the edges with even indices in $(h_1,\ldots,h_l)$ will either
  be $\{h_2,h_4,\ldots,h_l\}$ or $\{h_1,\ldots,h_{l-1}\}$).
  If $f = F_1$, then
  \[
    O_2 \cdot g_1 + E_1 \cdot f_{p'} = F_1 \cdot f_p \text{.}
  \]
  On the other hand, if $f = F_2$, then
  \[
  E_2 \cdot g_1 + O_1 \cdot f_{p'} = F_2 \cdot f_p\text{.}
  \]
    
  If $l$ is odd, we have $g_1 = O_1F_2 - E_1F_1 \in I_{G_{1}}$,
  $f_{p'} = E_2F_2 - O_2F_1$, and $f = F_2$ (since the only edges with
  even indices in $(h_1,\ldots,h_l)$ are $\{h_2,\ldots,h_{l-1}\}$).
  Furthermore, we have
  \[
    O_2 \cdot g_1 - E_2 \cdot f_{p'} = F_2 \cdot f_p \text{.}
  \]
  Note that $p'$ in both the odd and even case is a path in $G$
  accepting a representation as in Equation \ref{eq:subpaths_G1}
  with exactly one less subpath $p'_1, \ldots p'_{u-1}$. Hence,
  the statement follows by the induction hypothesis.
\end{proof}

\begin{remark} By exploiting the characterization
  of primitive even closed walks (see \cite[Lemma 5.11]{Binomi}), one
  can replace the saturation of $f$ in
  Theorem \ref{thm:glue-Gs-along-path} with the second power,
  that is, $I_G = (I_{G_1}+I_{G_2}):f^2$.  For the purposes of this paper
  we only require the saturation, so we have elected not to present the
  more technical proof.
\end{remark}

\begin{example}
 Theorem \ref{thm:glue-Gs-along-path} is false if we drop the assumption that at least one graph is bipartite.
  Clearly the toric ideal of a triangle is the zero ideal. Suppose we glue two triangles along an edge (see Figure \ref{fig:not-bipart}). Then the toric ideal of the resulting graph will be nontrivial since there is now a four cycle, introducing a nonzero generator.
  \begin{figure}[!ht]
    \centering
    \begin{tikzpicture}[scale=.36]
      \coordinate (v1) at (0,0);
      \coordinate (v2) at (3,3);
      \coordinate (v3) at (0,3);

      \coordinate (v4) at (1,0);
      \coordinate (v5) at (4,0);
      \coordinate (v6) at (4,3);

      \coordinate (v7) at (7,0);
      \coordinate (v8) at (10,0);
      \coordinate (v9) at (10,3);
      \coordinate (v10) at (7,3);

      \draw[very thick] (v1)--(v2);
      \draw (v2)--(v3);
      \draw (v3)--(v1);

      \draw (v4)--(v5);
      \draw (v5)--(v6);
      \draw[very thick] (v6)--(v4);

      \draw (v7)--(v8);
      \draw (v8)--(v9);
      \draw (v9)--(v10);
      \draw (v7)--(v10);
      \draw[very thick] (v7)--(v9);

      \draw[thick,-latex] (4.5,1.5) -- (6.5,1.5);

      \fill[fill=white,draw=black] (v1) circle (.1);
      \fill[fill=white,draw=black] (v2) circle (.1);
      \fill[fill=white,draw=black] (v3) circle (.1);
      \fill[fill=white,draw=black] (v4) circle (.1);
      \fill[fill=white,draw=black] (v5) circle (.1);
      \fill[fill=white,draw=black] (v6) circle (.1);
      \fill[fill=white,draw=black] (v7) circle (.1);
      \fill[fill=white,draw=black] (v8) circle (.1);
      \fill[fill=white,draw=black] (v9) circle (.1);
      \fill[fill=white,draw=black] (v10) circle (.1);
    \end{tikzpicture}
    \caption{Theorem \ref{thm:glue-Gs-along-path} is false if none of the graphs are bipartite.}
    \label{fig:not-bipart}
  \end{figure}
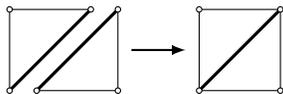
When we glue together non-bipartite graphs, we may introduced new primitive even walks, i.e., generators.
\end{example}

If the path in Theorem \ref{thm:glue-Gs-along-path} has length one,
i.e, it is an edge, we get a splitting of $I_G$.

\begin{corollary}\label{t. main splitting}
  Let $G$ be a graph, and suppose that $G_1$ and $G_2$ form a splitting of $G$ along an edge $e$. If $G_1$ is bipartite, then $I_G=I_{G_1}+I_{G_2}$.
\end{corollary}

\begin{remark}
If we view $G$ as the glued graph of
$G_1$ and $G_2$, then note that 
Corollary \ref{t. main splitting} 
does not depend upon
the orientation of the gluing, i.e.,
it is independent of the graph
isomorphism $\varphi$.  However,
this fact requires that $G_1$ is bipartite.  If we glue two non-bipartite graphs
along an edge, then, as noted 
in Example \ref{e.gluing}, the 
resulting  graphs are non-isomorphic.  
In fact,
the toric ideals of the resulting graphs
may not be equal.  For example, the toric
ideals of the two graphs in Figure
\ref{f. gluing edges2} will have
non-equal toric ideals.
\end{remark}

Note that Corollary \ref{t. main splitting} is false if we split a graph along a path of length $>1$:

\begin{example}\label{e. generaliz 1 does not work}
  Let $G$ be the graph in Figure \ref{f. glued graph on an even path}. Note that the two subsets $W_1 = \{ x_1,x_2,x_3,x_4\}$ and $W_2 = \{ x_1, x_2, x_3,y_4 \}$ of $V(G)$ yield two induced subgraphs $G_1, G_2$ which intersect along the diagonal $H \cong P_2$ which form a splitting of $G$.
  \begin{figure}[!ht]
    \centering
    \begin{tikzpicture}[scale=0.5]
      % define points
      \coordinate (v1) at (0,0);
      \coordinate (v2) at (3,0);
      \coordinate (v3) at (3,3);
      \coordinate (v4) at (0,3);
      \coordinate (v5) at (1.5,1.5);

      % draw edges
      \draw (v1)--(v2) node[midway,below]{$\scriptstyle e_5$};
      \draw (v2)--(v3) node[midway,right]{$\scriptstyle e_6$};
      \draw (v3)--(v4) node[midway,above]{$\scriptstyle e_3$};
      \draw (v4)--(v1) node[midway,left]{$\scriptstyle e_4$};
      \draw (v1)--(v5) node[midway,right]{$\scriptstyle e_1$};
      \draw (v3)--(v5) node[midway,left]{$\scriptstyle e_2$};
		
      % draw graph nodes and their labels	
      \fill[fill=white,draw=black] (v1) circle (0.1) node[below left]{$x_1$};
      \fill[fill=white,draw=black] (v2) circle (0.1) node[below right]{$y_4$};
      \fill[fill=white,draw=black] (v3) circle (0.1) node[above right]{$x_3$};
      \fill[fill=white,draw=black] (v4) circle (0.1) node[above left] {$x_4$};
      \fill[fill=white,draw=black] (v5) circle (0.1) node[right=1mm]{$x_2$};
    \end{tikzpicture}
    \caption{The graph from Example \ref{e. generaliz 1 does not work}.}
    \label{f. glued graph on an even path}
  \end{figure}
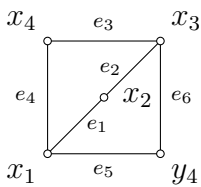	
  Since $G_i$ is isomorphic to a four cycle, its toric ideal $I_{G_i}$ is generated by a single generator, i.e., $I_{G_1} = \langle e_1e_3 - e_2e_4 \rangle$ and $I_{G_2} = \langle e_2e_5-e_1e_6 \rangle$. However, the toric ideal $I_G$ of $G$ has three generators corresponding to three primitive even closed walks of length four, namely $I_G = \langle e_1e_3 - e_2e_4, e_2e_5 - e_1e_6, e_3e_5 - e_4e_6 \rangle$.  Hence $I_G \neq I_{G_1} +I_{G_2}$.
  However, $e_2(e_3e_5 - e_4e_6) = e_6(e_1e_3-e_2e_4)+e_3(e_2e_5-e_1e_6)$,
  so $e_3e_5 - e_4e_6 \in  (I_{G_1}+I_{G_2}):e_2^\infty$
  \end{example}

The toric splittings of $I_G$ in Corollary \ref{t. main splitting}
has consequences for the Betti numbers of $I_G$.

\begin{theorem}\label{t.tensor product}
  Let $G = (V(G), E(G))$ be a graph. Suppose that $G_1=(V(G_{1}),E(G_{1})), G_2=(V(G_{2}),E(G_{2}))$ are two induced subgraphs which form a splitting of $G$ along an edge $e$. If $G_1$ is bipartite, then
  \[
    \beta_{i,j}(\K[E(G)]/I_G)=\sum_{\substack{i_1+i_2 = i\\j_1+j_2 = j}} {\beta_{i_1,j_1}(\K[E(G_{1})]/I_{G_1})\beta_{i_2,j_2}( \K[E(G_{2})]/I_{G_2})} ~~\mbox{for all $i,j \geq 0$.}
  \]
\end{theorem}

\begin{proof}
  Splitting and gluing of graphs are inverse operations. To be more precise, $G$ can be obtained from the disjoint union of $G_1$ and $G_2$ and then
  identifying the corresponding edge in $G_1$, respectively $G_2$.
  We translate this graph theoretical construction into algebra.

  Let $G'_i = (V(G_{i}'), E(G_{i}'))$ be isomorphic to $G_i$ where we assume that $V(G_{1}') \cap V(G_{2}') = \emptyset$. We define the graph   $G'=(V(G')=V(G_{1}')\sqcup V(G_{2}'),E(G')=E(G_{1}')\sqcup E(G_{2}'))$. Let $e'_i \in E(G_{i}')$ be the edges along which we glue. Algebraically, the process of gluing $e'_1$ along $e'_2$ corresponds to taking the quotient by the principal ideal $(e'_1-e'_2) \subseteq \K[E(G')]$. By Corollary \ref{t. main splitting}, $I_G$ corresponds under the isomorphism $\K[E(G)] \cong \K[E(G')]/(e'_1-e'_2)$ to $(I_{G'} + (e'_1-e'_2))/(e'_1-e'_2)$, so that we obtain:
  \[
  \K[E(G)]/I_G \cong \frac{\K[E(G')]/(e'_1-e'_2)}{(I_{G'} + (e'_1-e'_2))/(e'_1-e'_2)} \cong \frac{\K[E(G')]}{I_{G'} + (e'_1-e'_2)} \text{.}
  \]
  The toric ideal of $G'$ is a prime ideal, so $\K[E(G')]/I_{G'}$ is a domain. Therefore ${e'_1-e'_2}$ gives rise to a regular form in $\K[E(G')]/I_{G'}$. As a consequence of \cite[Corollary 20.4]{PeevaBook}, both $\K[E(G')]/I_{G'}$ and $\K[E(G')]/(I_{G'}+(e'_1-e'_2))$ share the same graded Betti numbers. But $I_{G'} = I_{G'_1} + I_{G'_2} \subseteq \K[E(G')]$
  where each ideal is 
  in a different set of variables. So then
  \[
    \K[E(G')]/I_{G'} = \K[E(G_{1}')]/I_{G'_1} \otimes_\K \K[E(G'_{2})]/I_{G'_2} \cong \K[E(G_{1})]/I_{G_1} \otimes_\K \K[E(G_{2})]/I_{G_2} \text{.}
  \]
  By taking the tensor product of the resolutions
  of $\K[E(G_{1})]/I_{G_1}$ and $\K[E(G_{2})]/I_{G_2}$, we have
  \begin{eqnarray*}
    \beta_{i,j}(\K[E(G)]/I_G) &= &\beta_{i,j}(\K[E(G')]/(I_{G'}+(e'_1-e'_2)) \\
    & = & \beta_{i,j}(\K[E(G')]/I_{G'}) = 
    \sum_{\substack{i_1+i_2 = i\\j_1+j_2 = j}} {\beta_{i_1,j_1}(\K[E(G_{1})]/I_{G_1})\beta_{i_2,j_2}( \K[E(G_{2})]/I_{G_2})},
  \end{eqnarray*}
  as desired.
\end{proof}

\begin{remark}
Given graphs $G_1,\ldots,G_n$, where
at most one graph is not bipartite,
one can first glue $G_1$ and $G_2$
along an edge to form
$G_{1,2}$, then glue $G_3$ along an
edge of $G_{1,2}$ to form $G_{1,2,3}$,
and so on, to form a new graph
$G_{1,2,\ldots,n}$. By iteratively applying the results in this section,
we can compute
the graded Betti numbers of this
new graph.  
Theorem \ref{t. Splitting 1} can be seen as a special case of what has just been remarked, where the first graph
is an even cycle $C$, and then 
we glue the remaining graphs along
edges of $C$ (see Figure \ref{fig:combined}). 
\end{remark}

\begin{example}\label{e. of tensor product betti}
 To illustrate some of the ideas of this
 section, consider the three graphs in Figure \ref{f. 4 cycles glued at an edge}. 
  \begin{figure}[!ht]
    \centering
    \begin{tikzpicture}[scale=.4]
      \node (f1) at (-1.5,2) {$G=\ $};

      % define points
      \coordinate (v1) at (0,0);
      \coordinate (v2) at (3,0);
      \coordinate (v3) at (3,3);
      \coordinate (v4) at (0,3);
      \coordinate (v5) at (4.5,1);
      \coordinate (v6) at (4.5,4);
      \coordinate (v7) at (6,0);
      \coordinate (v8) at (6,3);
      \coordinate (v9) at (1.5,-1);
      \coordinate (v10) at (1.5,2);

      % draw edges
      \draw (v1)--(v2);
      \draw (v2)--(v3);
      \draw (v3)--(v4);
      \draw (v4)--(v1);
      \draw (v2)--(v5);
      \draw (v3)--(v6);
      \draw (v5)--(v6);
      \draw (v2)--(v7);
      \draw (v3)--(v8);
      \draw (v7)--(v8);
      \draw (v2)--(v9);
      \draw (v3)--(v10);
      \draw (v9)--(v10);

      % draw nodes
      \fill[fill=white,draw=black] (v1) circle (0.1);
      \fill[fill=white,draw=black] (v2) circle (0.1);
      \fill[fill=white,draw=black] (v3) circle (0.1);
      \fill[fill=white,draw=black] (v4) circle (0.1);
      \fill[fill=white,draw=black] (v5) circle (0.1);
      \fill[fill=white,draw=black] (v6) circle (0.1);
      \fill[fill=white,draw=black] (v7) circle (0.1);
      \fill[fill=white,draw=black] (v8) circle (0.1);
      \fill[fill=white,draw=black] (v9) circle (0.1);
      \fill[fill=white,draw=black] (v10) circle (0.1);
    \end{tikzpicture}
    \begin{tikzpicture}[scale=0.4]
      \node (f1) at (-1.5,1) {$G'=\ $};

      \coordinate (v1) at (0,0);
      \coordinate (v2) at (2,0);
      \coordinate (v3) at (2,2);
      \coordinate (v4) at (0,2);
      \coordinate (v5) at (4,0);
      \coordinate (v6) at (4,2);
      \coordinate (v7) at (4,4);
      \coordinate (v8) at (2,4);
      \coordinate (v9) at (0,-2);
      \coordinate (v10) at (2,-2);

      \draw (v1)--(v2);
      \draw (v2)--(v3);
      \draw (v3)--(v4);
      \draw (v4)--(v1);
      \draw (v2)--(v5);
      \draw (v3)--(v6);
      \draw (v5)--(v6);
      \draw (v6)--(v7);
      \draw (v3)--(v8);
      \draw (v7)--(v8);
      \draw (v1)--(v9);
      \draw (v3)--(v10);
      \draw (v9)--(v10);

      \fill[fill=white,draw=black] (v1) circle (0.1);
      \fill[fill=white,draw=black] (v2) circle (0.1);
      \fill[fill=white,draw=black] (v3) circle (0.1);
      \fill[fill=white,draw=black] (v4) circle (0.1);
      \fill[fill=white,draw=black] (v5) circle (0.1);
      \fill[fill=white,draw=black] (v6) circle (0.1);
      \fill[fill=white,draw=black] (v7) circle (0.1);
      \fill[fill=white,draw=black] (v8) circle (0.1);
      \fill[fill=white,draw=black] (v9) circle (0.1);
      \fill[fill=white,draw=black] (v10) circle (0.1);
    \end{tikzpicture}
    \begin{tikzpicture}[scale=0.4]
      \node (f1) at (-1.5,3) {$G'' =\ $};

      % define points
      \coordinate (v1) at (0,0);
      \coordinate (v2) at (2,0);
      \coordinate (v3) at (2,2);
      \coordinate (v4) at (0,2);
      \coordinate (v5) at (4,0);
      \coordinate (v6) at (4,2);
      \coordinate (v7) at (4,4);
      \coordinate (v8) at (2,4);
      \coordinate (v9) at (0,4);

      % draw edges
      \draw (v1)--(v2);
      \draw (v2)--(v3);
      \draw (v3)--(v4);
      \draw (v4)--(v1);
      \draw (v2)--(v5);
      \draw (v3)--(v6);
      \draw (v5)--(v6);
      \draw (v6)--(v7);
      \draw (v3)--(v8);
      \draw (v7)--(v8);
      \draw (v4)--(v9);
      \draw (v9)--(v8);

      % draw nodes
      \fill[fill=white,draw=black] (v1) circle (0.1);
      \fill[fill=white,draw=black] (v2) circle (0.1);
      \fill[fill=white,draw=black] (v3) circle (0.1);
      \fill[fill=white,draw=black] (v4) circle (0.1);
      \fill[fill=white,draw=black] (v5) circle (0.1);
      \fill[fill=white,draw=black] (v6) circle (0.1);
      \fill[fill=white,draw=black] (v7) circle (0.1);
      \fill[fill=white,draw=black] (v8) circle (0.1);
      \fill[fill=white,draw=black] (v9) circle (0.1);
    \end{tikzpicture}
    \caption{$G$ and $G'$ are a gluing of four $4$-cycles at an edge, $G''$ is not.}
    \label{f. 4 cycles glued at an edge} 
  \end{figure}
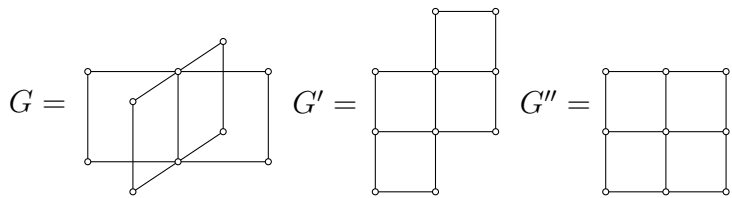
The graphs $G$ and $G'$ are obtaining by gluing in two different ways four copies of the four cycle $C_4$ along one edge. 
Note that it is not possible to construct $G''$ by iteratively gluing four four cycles along one edge at each step.
    From Theorem \ref{t.tensor product}, the ideals $I_G$ and $I_{G'}$ have the same graded Betti numbers, $\beta_{i,j}\coloneqq \beta_{i,j}(\K[E(G)]/I_G)=\beta_{i,j}(\K[E(G')]/I_{G'})$, see the following Betti table. One can check, using for instance \emph{Macaulay2}, that $I_{G''}$ has graded Betti numbers  $\beta_{i,j}''\coloneqq \beta_{i,j}(\K[E(G'')]/I_{G''})$ that are different from $\beta_{i,j}$ as seen in the second
    Betti table.
  \begin{center}{\tiny
    $\beta_{i,j} \coloneqq$ \begin{tabular}{cccccc}
      &0&1&2&3&4\\
      total:&1&4&6&4&1\\
      0: & 1 & . & . & . & .\\
      1: & . & 4 & . & . & . \\
      2: & . &.& 6 & . & .\\
      3: & . &.&.& 4 & .\\
      4: & . & . & . & . & 1
    \end{tabular}\ \ \
    $\beta''_{i,j}\coloneqq$
    \begin{tabular}{cccccc}
      & 0 & 1 & 2 & 3 & 4\\
      total: & 1 & 5 & 10 & 10 & 4\\
      0: & 1 & . & . & . & .\\
      1: & . & 4 & . & . & .\\
      2: & . & . & 6 & . & .\\
      3: & . & 1 & 4 & 10 & 4\\
    \end{tabular}
}
  \end{center}
\end{example}

%%%%%%%%%%%%%%%%%%%%%%%%%%%%%%%%%%%%%%%%%%%%%%%%%%%%%%%%%%%%%%%%%%%%%

\section{Future directions}
\label{sec:further-dir}

Theorem \ref{t. Splitting 1} and Corollary \ref{t. main splitting} describe
two ways in which the toric ideal of a graph can be split.  It is
natural to ask the following (but possibly difficult) question.

\begin{question}\label{q. other splitting}
  For what graphs $G$ can we find graphs $G_1$ and $G_2$ so that their
  respective toric ideals satisfy $I_G = I_{G_1} + I_{G_2}$?  More generally,
  can we classify when $I_G$ is a splittable toric ideal in terms
  of $G$?
\end{question}

In Theorem \ref{t. Splitting 1} and Corollary \ref{t. main splitting},
our graphs are glued along a single edge.  An edge can also be viewed
as a complete graph.  A \emph{complete graph} on $n$
vertices, denoted $K_n$,  is the graph where each vertex is adjacent to
every other vertex.  Since an edge is a $K_2$, it
is natural to ask if our main results can be generalized if we glue
along a subgraph that is a complete graph.

As an example of this behaviour, consider the two graphs $G_1$ and $G_2$
that are glued along the triangle (which is a $K_3$) to create the graph
$G$ as in Figure \ref{f. firstgluing_a}.  We have highlighted the glued
edges in $G$ by making the corresponding
edges thicker.
\begin{figure}[!ht]
    \centering
    \begin{tikzpicture}[scale=.5]
      % define vertices
      \coordinate (y1) at (0,0);
      \coordinate (y2) at (2,0);
      \coordinate (y3) at (1,-1);
      \coordinate (y4) at (.5,1);
      \coordinate (y5) at (1.5,1);
      \node (f1) at (1.2,-2) {$G_1$};
      \draw (y1) -- (y2);
      \draw (y1) -- (y4) -- (y5) -- (y2) -- (y3) -- (y1);

      \coordinate (z1) at (4,0);
      \coordinate (z2) at (6,0);
      \coordinate (z3) at (8,0);
      \coordinate (z4) at (10,0);
      \coordinate (z5) at (11,1);
      \coordinate (z6) at (11,-1);
      \coordinate (z7) at (5,-1);
       \node (f2) at (8.2,-2) {$G_2$};

      \draw (z1) -- (z2) -- (z4) -- (z3) -- (z4) -- (z5) -- (z6) -- (z4);
      \draw (z1) -- (z7) -- (z2);
      
      \coordinate (x1) at (13,0);
      \coordinate (x2) at (15,0);
      \coordinate (x3) at (17,0);
      \coordinate (x4) at (19,0);
      \coordinate (x5) at (20,1);
      \coordinate (x6) at (20,-1);
      \coordinate (x7) at (14,-1);
      \coordinate (x8) at (13.5,1);
      \coordinate (x9) at (14.5,1);
      \node (f3) at (17.2,-2) {$G$};

      % draw edges
      \draw[very thick] (x1) -- (x2);
      \draw (x2) -- (x3);
      \draw (x3) -- (x4);
      \draw (x4) -- (x5);
      \draw (x5) -- (x6);
      \draw (x6) -- (x4);
      \draw[very thick] (x1) -- (x7);
      \draw[very thick] (x7) -- (x2);
      \draw (x1) -- (x8);
      \draw (x8) -- (x9);
      \draw (x9) -- (x2);

      % draw vertices
      \fill[fill=white,draw=black] (x1) circle (.1);
      \fill[fill=white,draw=black] (x2) circle (.1);
      \fill[fill=white,draw=black] (x3) circle (.1);
      \fill[fill=white,draw=black] (x4) circle (.1);
      \fill[fill=white,draw=black] (x5) circle (.1);
      \fill[fill=white,draw=black] (x6) circle (.1);
      \fill[fill=white,draw=black] (x7) circle (.1);
      \fill[fill=white,draw=black] (x8) circle (.1);
      \fill[fill=white,draw=black] (x9) circle (.1);

      \fill[fill=white,draw=black] (y1) circle (.1);
      \fill[fill=white,draw=black] (y2) circle (.1);
      \fill[fill=white,draw=black] (y3) circle (.1);
      \fill[fill=white,draw=black] (y4) circle (.1);
      \fill[fill=white,draw=black] (y5) circle (.1);
      
      \fill[fill=white,draw=black] (z1) circle (.1);
      \fill[fill=white,draw=black] (z2) circle (.1);
      \fill[fill=white,draw=black] (z3) circle (.1);
      \fill[fill=white,draw=black] (z4) circle (.1);
      \fill[fill=white,draw=black] (z5) circle (.1);
      \fill[fill=white,draw=black] (z6) circle (.1);
      \fill[fill=white,draw=black] (z7) circle (.1);
    \end{tikzpicture}
    \caption{The graph $G$ obtained by gluing $G_1$ and $G_2$ along a $K_3$}
    \label{f. firstgluing_a}
\end{figure}
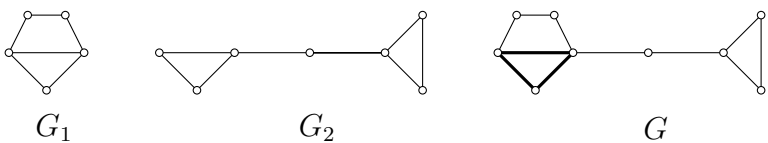
\noindent
Using a computer algebra system, one can verify
that the toric ideal of $I_G$ is splittable, and in fact,
$I_G = I_{G_1} + I_{G_2}.$

The graph $G$ actually highlights a subtlety of Question
\ref{q. other splitting} since the splitting of $I_G$ also follows
from our results.
In particular, observe that $G$ can also be constructed by
gluing the graphs $G_1'$ and $G_2$ along a single edge as in
Figure \ref{f. secondgluing}.
\begin{figure}[!ht]
    \centering
    \begin{tikzpicture}[scale=.5]
      % define vertices
      \coordinate (y1) at (0,0);
      \coordinate (y2) at (2,0);
      \coordinate (y4) at (.5,1);
      \coordinate (y5) at (1.5,1);
      \node (f1) at (1.2,-2) {$G'_1$};
      \draw (y1) -- (y2);
      \draw (y1) -- (y4) -- (y5) -- (y2) -- (y1);

      \coordinate (z1) at (4,0);
      \coordinate (z2) at (6,0);
      \coordinate (z3) at (8,0);
      \coordinate (z4) at (10,0);
      \coordinate (z5) at (11,1);
      \coordinate (z6) at (11,-1);
      \coordinate (z7) at (5,-1);
       \node (f2) at (8.2,-2) {$G_2$};

      \draw (z1) -- (z2) -- (z4) -- (z3) -- (z4) -- (z5) -- (z6) -- (z4);
      \draw (z1) -- (z7) -- (z2);
      
      \coordinate (x1) at (13,0);
      \coordinate (x2) at (15,0);
      \coordinate (x3) at (17,0);
      \coordinate (x4) at (19,0);
      \coordinate (x5) at (20,1);
      \coordinate (x6) at (20,-1);
      \coordinate (x7) at (14,-1);
      \coordinate (x8) at (13.5,1);
      \coordinate (x9) at (14.5,1);
      \node (f3) at (17.2,-2) {$G$};

      % draw edges
      \draw[very thick] (x1) -- (x2);
      \draw (x2) -- (x3);
      \draw (x3) -- (x4);
      \draw (x4) -- (x5);
      \draw (x5) -- (x6);
      \draw (x6) -- (x4);
      \draw (x1) -- (x7);
      \draw (x7) -- (x2);
      \draw (x1) -- (x8);
      \draw (x8) -- (x9);
      \draw (x9) -- (x2);

      % draw vertices
      \fill[fill=white,draw=black] (x1) circle (.1);
      \fill[fill=white,draw=black] (x2) circle (.1);
      \fill[fill=white,draw=black] (x3) circle (.1);
      \fill[fill=white,draw=black] (x4) circle (.1);
      \fill[fill=white,draw=black] (x5) circle (.1);
      \fill[fill=white,draw=black] (x6) circle (.1);
      \fill[fill=white,draw=black] (x7) circle (.1);
      \fill[fill=white,draw=black] (x8) circle (.1);
      \fill[fill=white,draw=black] (x9) circle (.1);

      \fill[fill=white,draw=black] (y1) circle (.1);
      \fill[fill=white,draw=black] (y2) circle (.1);
      \fill[fill=white,draw=black] (y4) circle (.1);
      \fill[fill=white,draw=black] (y5) circle (.1);
      
      \fill[fill=white,draw=black] (z1) circle (.1);
      \fill[fill=white,draw=black] (z2) circle (.1);
      \fill[fill=white,draw=black] (z3) circle (.1);
      \fill[fill=white,draw=black] (z4) circle (.1);
      \fill[fill=white,draw=black] (z5) circle (.1);
      \fill[fill=white,draw=black] (z6) circle (.1);
      \fill[fill=white,draw=black] (z7) circle (.1);
    \end{tikzpicture}
    \caption{The graph $G$ obtained by gluing $G_1'$ and $G_2$ along a $K_2$}
    \label{f. secondgluing}
\end{figure}
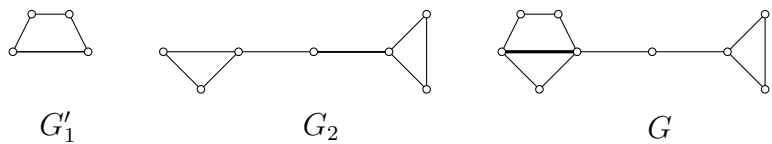
Since $G'_1$ is bipartite, Corollary \ref{t. main splitting} gives
$I_G = I_{G'_1} + I_{G_2}$.  Note that $I_{G'_1} = I_{G_1}$ since the
non-bipartite graph $G_1$ has only one generator coming from the four cycle.
Thus, the two splittings are the same.

Since we are interested in the Betti numbers
of $I_G$, we pose a follow up to Question
\ref{q. other splitting}.

\begin{question}\label{q. betti numbers}
  Suppose that there exists graphs $G, G_1$ and $G_2$ such that
  $I_G = I_{G_1} + I_{G_2}$.  How do the graded Betti numbers of
  $I_G$ related to those of $I_{G_1}$ and $I_{G_2}$?
  \end{question}

Understanding Questions
\ref{q. other splitting} and \ref{q. betti numbers} for arbitrary toric
ideals would also be of interest.


\begin{thebibliography}{10}

\bibitem{AH}
A.~Aramova and J.~Herzog.
\newblock Koszul cycles and {E}liahou-{K}ervaire type resolutions.
\newblock {\em J. Algebra}, 181(2):347--370, 1996.

\bibitem{BOVT}
J.~Biermann, A.~O'Keefe, and A.~Van~Tuyl.
\newblock Bounds on the regularity of toric ideals of graphs.
\newblock {\em Adv. in Appl. Math.}, 85:84--102, 2017.

\bibitem{CM}
A.~Campillo and C.~Marijuan.
\newblock Higher order relations for a numerical semigroup.
\newblock {\em S\'{e}m. Th\'{e}or. Nombres Bordeaux (2)}, 3(2):249--260, 1991.

\bibitem{CP}
A.~Campillo and P.~Pis\'{o}n.
\newblock L'id\'{e}al d'un semi-groupe de type fini.
\newblock {\em C. R. Acad. Sci. Paris S\'{e}r. I Math.}, 316(12):1303--1306, 1993.

\bibitem{CN}
A.~Corso and U.~Nagel.
\newblock Monomial and toric ideals associated to {F}errers graphs.
\newblock {\em Trans. Amer. Math. Soc.}, 361(3):1371--1395, 2009.

\bibitem{CoxLittleSchecnk}
D.~A. Cox, J.~B. Little, and H.~K. Schenck.
\newblock {\em Toric varieties}, volume 124 of {\em Graduate Studies in Mathematics}.
\newblock American Mathematical Society, Providence, RI, 2011.

\bibitem{DS}
P.~Diaconis and B.~Sturmfels.
\newblock Algebraic algorithms for sampling from conditional distributions.
\newblock {\em Ann. Statist.}, 26(1):363--397, 1998.

\bibitem{EH}
V.~Ene and J.~Herzog.
\newblock {\em Gr\"{o}bner bases in commutative algebra}, volume 130 of {\em Graduate Studies in Mathematics}.
\newblock American Mathematical Society, Providence, RI, 2012.

\bibitem{FHVT}
C.~A. Francisco, H.~T.~H\`a, and A.~Van~Tuyl.
\newblock Splittings of monomial ideals.
\newblock {\em Proc. Amer. Math. Soc.}, 137(10):3271--3282, 2009.

\bibitem{GHKKVTP}
F.~Galetto, J.~Hofscheier, G.~Keiper, C.~Kohne, A.~Van~Tuyl, and M.~E.~Uribe-Paczka.
\newblock Betti numbers of toric ideals of graphs: a case study.
\newblock {\em J. Algebra Appl.}, 18(12):1950226, 14, 2019.

\bibitem{GV}
I.~Gitler and C.~E.~Valencia.
\newblock Multiplicities of edge subrings.
\newblock {\em Discrete Math.}, 302(1-3):107--123, 2005.

\bibitem{M2}
D.~R. Grayson and M.~E. Stillman.
\newblock Macaulay2, a software system for research in algebraic geometry.
\newblock Available at \url{https://faculty.math.illinois.edu/Macaulay2/}.

\bibitem{GM}
Z.~Greif and J.~McCullough.
\newblock Green-{L}azarsfeld condition for toric edge ideals of bipartite graphs.
\newblock {\em J. Algebra}, 562:1--27, 2020.

\bibitem{HBOK}
H.~T.~H\`a, S.~K.~Beyarslan, and A.~O'Keefe.
\newblock Algebraic properties of toric rings of graphs.
\newblock {\em Comm. Algebra}, 47(1):1--16, 2019.

\bibitem{Binomi}
J.~Herzog, T.~Hibi, and H.~Ohsugi.
\newblock {\em Binomial ideals}, volume 279 of {\em Graduate Texts in Mathematics}.
\newblock Springer, Cham, 2018.

\bibitem{HHKOK}
T.~Hibi, A.~Higashitani, K.~Kimura, and A.~B.~O'Keefe.
\newblock Depth of edge rings arising from finite graphs.
\newblock {\em Proc. Amer. Math. Soc.}, 139(11):3807--3813, 2011.

\bibitem{HMT}
T.~Hibi, K.~Matsuda, and A.~Tsuchiya.
\newblock Edge rings with {$3$}-linear resolutions.
\newblock {\em Proc. Amer. Math. Soc.}, 147(8):3225--3232, 2019.

\bibitem{JK}
S.~Jacques and M.~Katzman.
\newblock The {B}etti numbers of forests.
\newblock {\em Preprint}, 2005.

\bibitem{KT}
K.~M. Koh and K.~L. Teo.
\newblock The search for chromatically unique graphs.
\newblock {\em Graphs Combin.}, 6(3):259--285, 1990.

\bibitem{NN}
R.~Nandi and R.~Nanduri.
\newblock Betti numbers of toric algebras of certain bipartite graphs.
\newblock {\em J. Algebra Appl.}, 18(12):1950231, 18, 2019.

\bibitem{OH}
H.~Ohsugi and T.~Hibi.
\newblock Toric ideals generated by quadratic binomials.
\newblock {\em J. Algebra}, 218(2):509--527, 1999.

\bibitem{PeevaBook}
I.~Peeva.
\newblock {\em Graded syzygies}, volume~14 of {\em Algebra and Applications}.
\newblock Springer-Verlag London, Ltd., London, 2011.

\bibitem{St}
B.~Sturmfels.
\newblock {\em Gr\"{o}bner bases and convex polytopes}, volume~8 of {\em University Lecture Series}.
\newblock American Mathematical Society, Providence, RI, 1996.

\bibitem{TT}
C.~Tatakis and A.~Thoma.
\newblock On the universal {G}r\"{o}bner bases of toric ideals of graphs.
\newblock {\em J. Combin. Theory Ser. A}, 118(5):1540--1548, 2011.

\bibitem{T}
A.~Tsuchiya.
\newblock {Edge rings of bipartite graphs with linear resolutions}.
\newblock {\em To appear J. Algebra Appl.}, 2020.
\newblock{\url{https://doi.org/10.1142/S0219498821501632}}

\bibitem{VRees}
R.~H.~Villarreal.
\newblock Rees algebras of edge ideals.
\newblock {\em Comm. Algebra}, 23(9):3513--3524, 1995.

\bibitem{V}
R.~H.~Villarreal.
\newblock {\em Monomial algebras}.
\newblock Monographs and Research Notes in Mathematics. CRC Press, Boca Raton, FL, second edition, 2015.

\end{thebibliography}
\end{document}